\documentclass[draftcls,12pt,onecolumn]{IEEEtran}
\usepackage{amsmath, amsfonts,amssymb, latexsym,epsfig,color,rotating,paralist,
times,float,subfigure,algorithm,verbatim,framed}
\usepackage{epstopdf}

\newcommand{\uC}{\underline{C}}
\newcommand{\umu}{\underline{\mu}}
\newcommand{\bmu}{\bar{\mu}}

\newcommand{\uV}{\underline{V}}

\newcommand{\mus}{{\mu^*(\model)}}
\newcommand{\rmm}{\rho_{\model,\bmodel}}
\newcommand{\Stop}{\mathcal{S}}
\newcommand{\Stopb}{\bar{\mathcal{S}}}
\newcommand{\uStop}{\underline{\mathcal{S}}}

\newcommand{\yp}{y_{\pi;\model}^*}
\newcommand{\yps}{y_{\pi;\model,\bmodel}^*}
\newcommand{\ypsx}{y_{e_X;\model,\bmodel}^*}
\newcommand{\ytt}{y^*_{\model,\bmodel}}

\newcommand{\bpi}{\bar{\pi}}

\newcommand{\ybp}{y_{\bp;\model}^*}
\newcommand{\ypm}{y_{\pi;\bmodel}^*}
\newcommand{\Rp}{\mathcal{R}_{\pi;\model}}
\newcommand{\Rpc}{\mathcal{R}^c_{\pi;\model}}

\newcommand{\tcost}[1]{J_\mu^{(#1)}}
\newcommand{\Rbp}{\mathcal{R}_{\bp;\model}}
\newcommand{\Rpbc}{\mathcal{R}^c_{\pi;\bmodel}}

\newcommand{\traj}{\boldsymbol{\pi}}
\newcommand{\aB}{R}
\newcommand{\bd}{\succeq_B}
\newcommand{\bu}{\bar{u}}

\newcommand{\model}{\theta}
\newcommand{\bmodel}{\bar{\model}}
\newcommand{\bB}{\bar{B}}
\newcommand{\btp}{\bar{\tp}}

\newcommand{\D}{D}

\newcommand{\delayset}{\{\D_1,\D_2,\ldots,\D_L\}}

\newcommand{\delayseti}{\{1,2,\ldots,L\}}

\newcommand{\Delayset}{\{0 \text{ (announce change)} ,\D_1,\D_2,\ldots,\D_L\}}
\newcommand{\Delayseti}{\{0 \text{ (announce change)} ,1,2,\ldots, L\}}
\newcommand{\tpone}{\tp^{(1)}}
\newcommand{\tptwo}{\tp^{(2)}}
\newcommand{\tp}{A}
\newcommand{\f}{{f}}
\newcommand{\mc}{m}
\newcommand{\kstar}{k^*}
\newcommand{\changetime}{t^*}
\renewcommand{\time}{t}
\newcommand{\epoch}{\tau}

\newcommand{\pizero}{\pi_0}

\newcommand{\Ep}{\E^\mu_{\pi_0}}
\newcommand{\Epzero}{\E^\mu_{\pi_0}}

\newcommand{\Cb}{\bar{C}}
\newcommand{\Vb}{\bar{V}}

\newcommand{\F}{\mathcal{F}}

\newcommand{\Tp}{T}
\newcommand{\sigp}{\sigma}

\newcommand{\Mu}{\boldsymbol{\mu}}
\newcommand{\U}{\mathcal{U}}

\newcommand{\X}{\mathbb{X}}
\newcommand{\Y}{\mathbb{Y}}

\newcommand{\ones}{\mathbf{1}_X}

\newcommand{\tpi}{\tilde{\pi}}

\newcommand{\p}{\prime}

\newcommand{\E}                 {\Bbb{E}}
\renewcommand{\P}                 {\Bbb{P}}
\newcommand{\reals}{\mathbb{R}}
\newcommand{\I}{\Pi(X)}

\renewcommand{\u}{\{1,2\}}
\newcommand{\argmin}{\operatorname{argmin}}
\newcommand{\ole}{\stackrel{\triangle}{=}}

\newtheorem{theorem}{Theorem}
\newtheorem{corollary}{Corollary}

\newcommand{\bT}{\bar{T}}
\newcommand{\bsigma}{\bar{\sigma}}

\newtheorem{lemma}{Lemma}
\newtheorem{definition}{Definition}

\newcommand{\gr}{\geq_r}
\newcommand{\lr}{\leq_r}
\newcommand{\gs}{\geq_s}
\newcommand{\ls}{\leq_s}

\newcommand{\bp}{{\bar{\pi}}}

\newcommand{\beq}{\begin{equation}}
\newcommand{\eeq}{\end{equation}}
\newcommand{\qed} {{$\hfill\blacksquare$}}

\begin{document}

\title{When to look at a noisy Markov chain in sequential decision making if measurements are costly? 
}
\author{Vikram Krishnamurthy  {\em Fellow, IEEE} 
\thanks{Vikram Krishnamurthy is
 with the Department of Electrical and Computer
Engineering, University of British Columbia, Vancouver, V6T 1Z4, Canada. 
(email:  vikramk@ece.ubc.ca). This research was partially supported by NSERC, Canada.}}

\maketitle

\begin{abstract} A decision maker records measurements of a finite-state  Markov chain  corrupted by  noise. The goal is to decide when   the Markov chain hits  a specific target state. The decision maker can choose from a finite set of sampling intervals to pick the next time
to look at the Markov chain. 
The aim is to
 optimize an objective comprising of false alarm, delay cost and cumulative measurement sampling cost.
Taking more frequent measurements yields accurate estimates but incurs a higher measurement cost. Making an erroneous decision too soon incurs a false alarm penalty. Waiting too long to declare the target state
incurs a delay penalty.
What is the optimal sequential strategy for the decision maker?
The paper shows that under reasonable conditions,
the optimal strategy has the following intuitive structure:
when the Bayesian estimate (posterior distribution) of the Markov chain is away from the target state, look less frequently; while if the posterior is close to the target state, look more frequently. Bounds are derived for the optimal strategy.
Also the achievable optimal cost
of the sequential detector as a function of transition dynamics and observation distribution is analyzed. The sensitivity
of the optimal achievable cost to parameter variations is bounded in terms of the Kullback divergence.
To prove the results in this paper, novel stochastic dominance
results on the Bayesian filtering recursion are derived. The formulation in this paper  generalizes quickest
time change detection to consider optimal sampling and also yields useful results in sensor scheduling
(active sensing).

 \end{abstract}


\section{Introduction and Examples}  \label{sec:intro}

\subsection{The Problem}
Consider the following quickest detection optimal sampling problem which is a special case of the problem considered
in this paper.
 Let $\epoch_k, k=0,1,\ldots$ denote the time
 instants at which decisions to observe a noisy finite state Markov chain are made.
   As it accumulates measurements over time,
  a decision-maker needs to announce when   the Markov chain hits a specific absorbing target state. 
At each decision time  $\epoch_k$, the decision maker needs to pick its decision from the action set  $\U =\Delayset$ where
 \begin{compactitem}
\item  Decision $u_{k}=0$ made at time $\epoch_k$ corresponds to ``{\em announce the target state and stop}''. When this decision is made the problem terminates at time $\epoch_k$ with possibly a  false alarm penalty (if the Markov chain was not in the target state).
\item  Decision $u_{k} \in \delayset$ at time $\epoch_k$  corresponds to:
``{\em Look at noisy Markov chain next at time $\epoch_{k+1} = \epoch_k + u_k$.}" Here 
$\D_1 < \D_2 < \cdots < \D_L$ are fixed positive integers. They denote the set of possible time delays to  sample the Markov chain next.
\end{compactitem}
Given the  history  of past measurements and  decisions, 
how should the decision-maker choose its decisions $u$? Let $\changetime$ denote the time at which the Markov chain hits
that the absorbing target state and $\kstar$ denote the time at which the decision maker announces that the Markov chain has hit 
the target state.
The decision-maker considers the following costs:\\
(i) {\em False alarm penalty}: If $\kstar < \changetime$, i.e., the Markov  chain is not in the target state, but the decision-maker announces that the chain
has hit the target state, it pays a false alarm penalty $\f$.\\
(ii) {\em Delay penalty}: If $\kstar  > \changetime$, i.e.,  the Markov chain hits the target state and the decision-maker does not announce this, it pays a delay penalty $d$.
The decision maker continues to pay this delay penalty over time until it announces the target state has been reached.
\\
(iii) {\em Sampling cost}: At each  decision time $\tau_k$, the decision maker  looks at the noisy Markov chain and
 pays a measurement (sampling) cost $m$. 
 
 Suppose the Markov chain starts with initial distribution $\pi_0$ at time $0$.
  {\em What is the optimal sampling strategy $\mu^*$ for the decision-maker
 to minimize  the following  combination of the false alarm rate, delay penalty and measurement cost?} 
That is, determine $\mu^* = \inf_\mu J_\mu(\pizero) $ where 
 \beq J_\mu(\pizero) =   d \Ep\{(\kstar - \changetime)^+\} + \f\, \P^\mu_{\pi_0}(\kstar < \changetime)  +
 m \sum_{u=1}^L \sum_{k:\epoch_k \leq \kstar} \P^\mu_{\pi_0}(u_k = u)
\label{eq:ksintro}
\eeq
Here $\mu$ denotes a  stationary strategy of  the decision maker.  $\P^\mu$ and $\Ep$ are the probability measure and expectation
of the evolution of the observations and Markov state which are strategy dependent (These are defined formally
in Sec.\ref{sec:qdbasic}).
 Taking  frequent measurements yields accurate estimates but incurs a higher measurement cost. Making an erroneous decision too soon incurs a false alarm penalty. Waiting too long to declare the target state
incurs a delay penalty.

 \subsection{Context}
In the special case  when the change time $\changetime$ is 
geometrically-distributed (equivalently, the Markov chain has two states), action space
$\U = \{0 \text{ (announce change)},1 \text{ (continue)}\}$, measurement cost $m = 0$, then (\ref{eq:ksintro}) 
becomes the classical Kolmogorov--Shiryayev
 quickest detection problem \cite{Shi78,PH08}.
Our setup generalizes this  in the following non-trivial ways:\\
First, unlike quickest detection, there are now multiple
``continue" actions  $u\in \delayseti$ corresponding to different sampling  delays  $\delayset$. (In quickest detection there is only one continue action and one stop action).
Each of these ``continue'' actions  result
in different dynamics of the posterior distribution and incur different costs. Also, the measurement costs
can be state and action dependent. 
\\
  Second, allowing for the underlying Markov chain to have multiple states  facilitates 
modelling general phase-distributed (PH-distributed) change times
(compared to two state Markov chains that model geometric distributed change times).
As described in \cite{Neu89}, a PH-distributed change time can be modelled
as a multi-state Markov chain with an absorbing state.  The optimal detection of a PH-distributed change point is useful  since  PH-distributions form a dense subset for the set of all distributions; see \cite{Kri11} for quickest detection 
with PH-distributed change times.

\subsection{Main Results, Organization and Related Works}
This paper analyzes the structure of the optimal sampling strategy of the decision-maker.
 The problem is an instance of a 
partially observed Markov decision process (POMDP) \cite{Cas98}.
   In general, solving POMDPs and therefore determining the optimal strategy is computationally intractable (PSPACE hard \cite{PT87}). 
  However, returning to the example considered above,
intuition suggests that the following  strategy would be sensible  (recall that the action set  $\U =\Delayset$):
\begin{compactitem}
\item If the Bayesian posterior distribution estimate of the Markov chain (given past observations and decisions) is away from the target state,  look infrequently at the noisy Markov chain. i.e., pick a large sampling interval $D_u$. Since we are interested in detecting when the Markov chain
hits the target state, there is little point in incurring a measurement cost by looking at the Markov chain when its estimate suggests that it is far away from the target state.
\item If the posterior distribution is close to the target state, then pay a higher sampling cost and look more frequently at the 
noisy Markov chain, i.e., pick a small sampling interval
$D_u$.
\item If the posterior is sufficiently close to the
 target state, then announce the target state has been reached, i.e., choose action $u=0$.
\end{compactitem}
The key point is that such a strategy (choice of sampling interval $D_u$) is monotonically decreasing as the posterior distribution gets closer to the target state.
By using  stochastic dominance and   lattice programming analysis,
   this paper shows that under reasonable conditions, the optimal sampling strategy always has this monotone structure.
  Lattice programming was championed by \cite{Top98} and provides a  general set of sufficient conditions for the existence
of monotone strategies in stochastic control problems.  This area falls under the general
umbrella of monotone comparative statics that has witnessed remarkable interest in the area of economics \cite{Ath02}.
Our results  apply to general observation distributions (Gaussians, exponentials, Markov modulated Poisson, discrete
memoryless channels, etc) and multi-state Markov chains.

In  more detail, this paper  establishes the following structural results:
\\
{\bf (i)} For  two-state Markov chains  observed in noise,  since the elements of the two-dimensional posterior probability mass function add to 1, 
 it suffices to consider one element  of this posterior -- this element is a probability and lies in the interval $[0,1]$.
Theorems~\ref{thm:main1} and \ref{cor:qd} show that under reasonable conditions
the optimal sampling strategy of the decision-maker has a monotone structure in the posterior distribution.
The monotone structure of  Theorem~\ref{thm:main1} reduces
a function space optimization problem  (dynamic programming on the space of posterior distributions) to a finite dimensional optimization -- 
since a monotone strategy  with $L$ possible actions has at most $L-1$ thresholds in the space of posterior distributions.
The threshold values  can be estimated via simulation based stochastic approximation.
The  monotone structure holds even for large delay penalty and measurement cost that is 
 independent of the state.
  If satisfaction is viewed as the number of times the decision maker looks at the Markov chain, 
  Theorems~\ref{thm:main1} and \ref{cor:qd}    say that ``delayed satisfaction" is optimal. These theorems also directly apply to a measurement
  control model recently developed in \cite{BV12} as will be discussed in Sec.\ref{sec:geom}.
  \\
{\bf (ii)}  For general-state Markov chains (which can model PH-distributed change times) observed in noise,  
the posterior lies in a $X-1$ dimensional unit simplex.
Theorem \ref{thm:myopic} shows that the optimal decision-maker's sampling strategy can be under-bounded by a 
judiciously chosen myopic strategy on the unit simplex of posterior distributions. Therefore the myopic strategy forms an easily
computable rigorous lower bound
to the optimal strategy. Sufficient conditions are given for the myopic strategy to have a monotone structure with respect
to the monotone likelihood ratio stochastic order on the simplex.
Theorem \ref{cor:qd2} illustrates the result for quickest detection problems.\\
{\bf (iii)} 
How does the  optimal expected  sampling cost  vary
  with
 transition matrix and noise distribution?  Is it possible to order   these parameters
 such that the larger they are, the larger the optimal sampling cost? Such a result would allow us to compare the optimal
 performance of different sampling models, even though computing these is intractable
For general-state Markov chains  observed in noise, Theorem~\ref{thm:tmove}  examines how the cost achieved by the optimal sampling strategy varies with
 transition matrix (state dynamics) and observation matrix (noise distribution). In particular  dominance measures are introduced
 for the transition matrix and observation distribution (Blackwell dominance) that result in the optimal cost increasing with respect
 to this dominance order. Theorem~\ref{thm:tmove}  shows that  for  optimal sampling 
problems, certain PH-distributions for the change time result in larger total optimal cost compared to other 
distributions.
\\
{\bf  (iv)}  Theorem \ref{thm:sens} derives  sensitivity bounds on the total cost for optimal sampling with
 a mis-matched model. That is, when the optimal strategy computed for a specific sampling model is used for a different
 sampling model, Theorem \ref{thm:sens} gives an explicit bound on the performance degradation.  In particular, by elementary use
 of the Pinsker inequality \cite{CT06},  Theorem \ref{thm:sens} shows that the sensitivity is a linear function  of the Kullback-Leibler divergence between the two models.
 Also, the bounds are tight in the sense
 that if the difference between the two models goes to zero, so does the performance degradation.
 \\
{\bf  (v)} To prove the above results, several important stochastic dominance properties of the Bayesian filter are presented in Theorem \ref{thm:filter}. How does the posterior distribution computed by the Bayesian filter vary with observation, prior, transition matrix
and observation matrix? Is it possible to order these so that the posterior distribution increases with respect to this ordering?
 These results are of independent interest.
 The theorem gives sufficient conditions for the Bayesian filtering recursion to preserves the MLR (monotone likelihood ratio)
 stochastic order, and for the normalization measure to be  submodular. It also shows  that if starting with two different transition matrices but identical priors, then the optimal predictor with the larger transition matrix
 (in terms of the order introduced in (\ref{eq:mor})) MLR dominates the predictor with the smaller transition matrix.

{\em Related Works}: In this paper we consider sampling control with change detection. A related problem is measurement control where
at each time the decision is made  whether to take a measurement
or not. This is the subject of the recent paper\footnote{The author is very grateful to Dr.\ Venu Veeravalli of U.\ Illinois Urbana Champaign
for sharing the results in \cite{BV12} and several useful discussions}  \cite{BV12} which considers geometric-distributed change times (2-state Markov chain).
The problem in \cite{BV12} can be formulated in terms of  our optimal sampling problem.  
 We discuss this further in Sec.\ref{sec:2state}. 

We also refer to the seminal work of Moustakides (see \cite{YMW12} and references therein) in event triggered sampling. 
Quickest detection has been studied widely, see \cite{PH08,TV05} and references therein.
We have considered recently  a POMDP approach to quickest detection with social learning \cite{Kri12} and non-linear penalties \cite{Kri11} 
and phase-distributed change times. However,
in these papers, there is only one continue and one stop action. The results in the current paper are considerably more general due
to the propagation of different dynamics for the multiple continue actions. A useful feature of the lattice programming approach  \cite{Alb79,Lov87,Rie91}
used in this paper is that the results apply to general observation noise distributions (Gaussians, exponentials, discrete
memoryless channels)  and multiple state Markov chains. Also,  the results proved here are valid for finite sample
sizes and no asymptotic approximations in signal to noise ratio are used.

\subsection{Examples: Change Detection and Sensor Scheduling} 
Several examples in statistical signal processing are special cases of the above measurement-sampling control model.  The terms active/smart/cognitive  
sensing imply the use of feedback of previous estimates and decisions to choose the current optimal decision. 

{\em Example 1. Quickest Time Change Detection with Optimal Sampling}:
Return to the problem  considered at the beginning of this section. The action space is $\{0 \text{ (announce change)}, D_1=1,D_2=3,D_3=5,D_4=10\}$.
That is, at each decision time, the decision  maker has the option of either stopping  or looking at a 2-state Markov chain every 1, 3, 5 or 10
time points.
 Suppose the decision maker observes the underlying Markov chain
via a binary erasure channel (parameters values are specified in Sec.\ref{sec:numerical}). 

Theorem \ref{thm:main1} shows
that 
 the optimal strategy has a monotone structure in posterior $\pi(1)$ depicted in 
Figure \ref{fig:thm1}(a).
   The horizontal axis in Figure \ref{fig:thm1}(a) denotes the Bayesian
 posterior $\pi(1)$  while the vertical axis denotes the optimal action taken. 
Therefore, when the posterior is less than $\pi_4^*$, it is optimal to look every 10 time points at the noisy Markov chain, for posterior in the interval $[\pi_4^*,\pi_3^*]$
look every 5 points at the noisy Markov chain, etc.
 Thus one only needs to compute/estimate the threshold values $\pi_1^*, \pi_2^*, \pi_3^*,\pi_4^*$ to determine the optimal strategy.
  The usefulness of Theorem \ref{thm:main1} 
 is further enhanced by noting
that  in general (without introducing conditions) the optimal strategy does not have this property.
Figure \ref{fig:thm1}(b) gives an example where
 the sufficient conditions of Theorem \ref{thm:main1} are violated and
  the optimal strategy is no longer monotone.

\begin{figure}\centering
\mbox{\subfigure[Structured Strategy]{\epsfig{figure=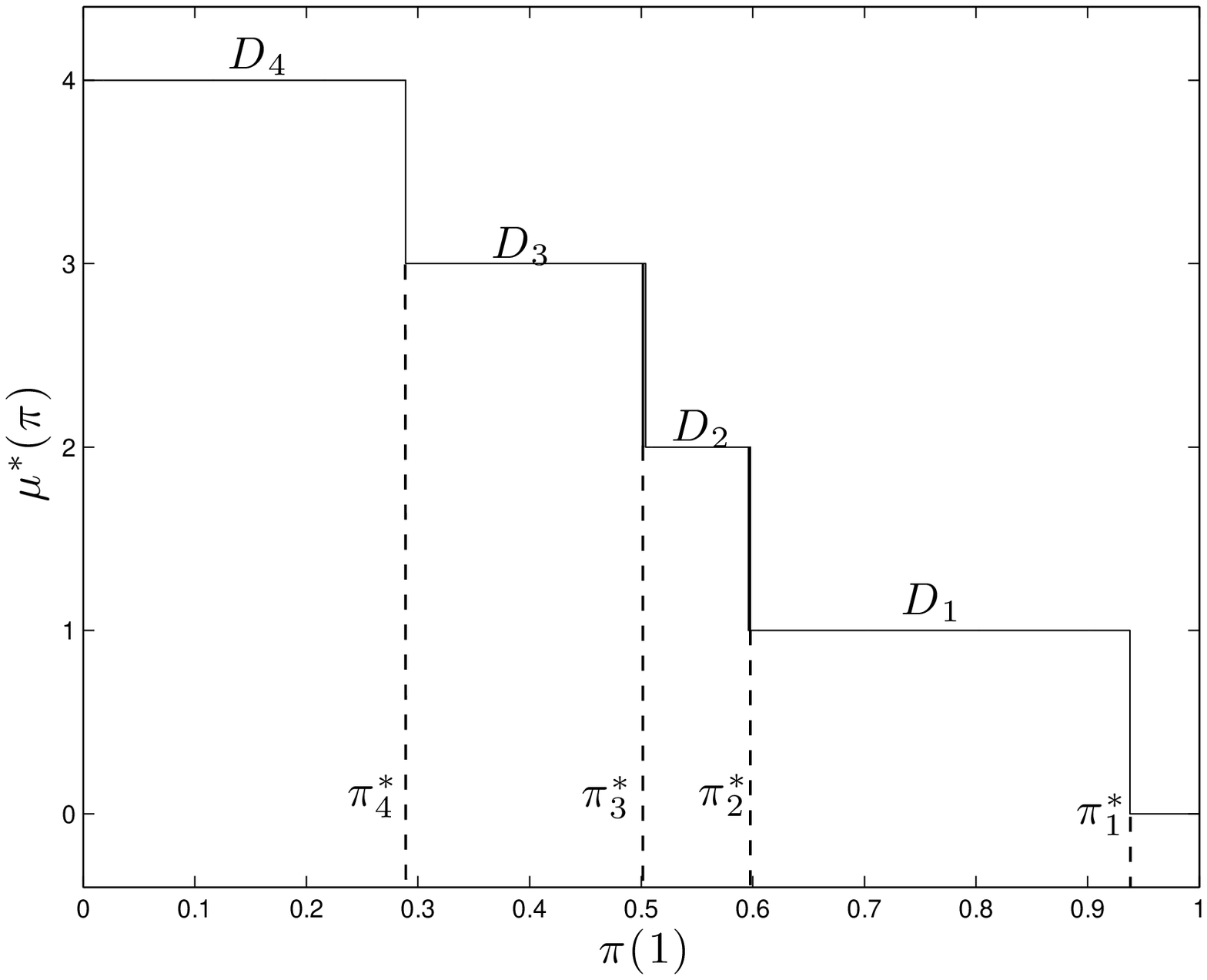,width=0.35\linewidth} } \quad
\subfigure[Unstructured Strategy]{\epsfig{figure=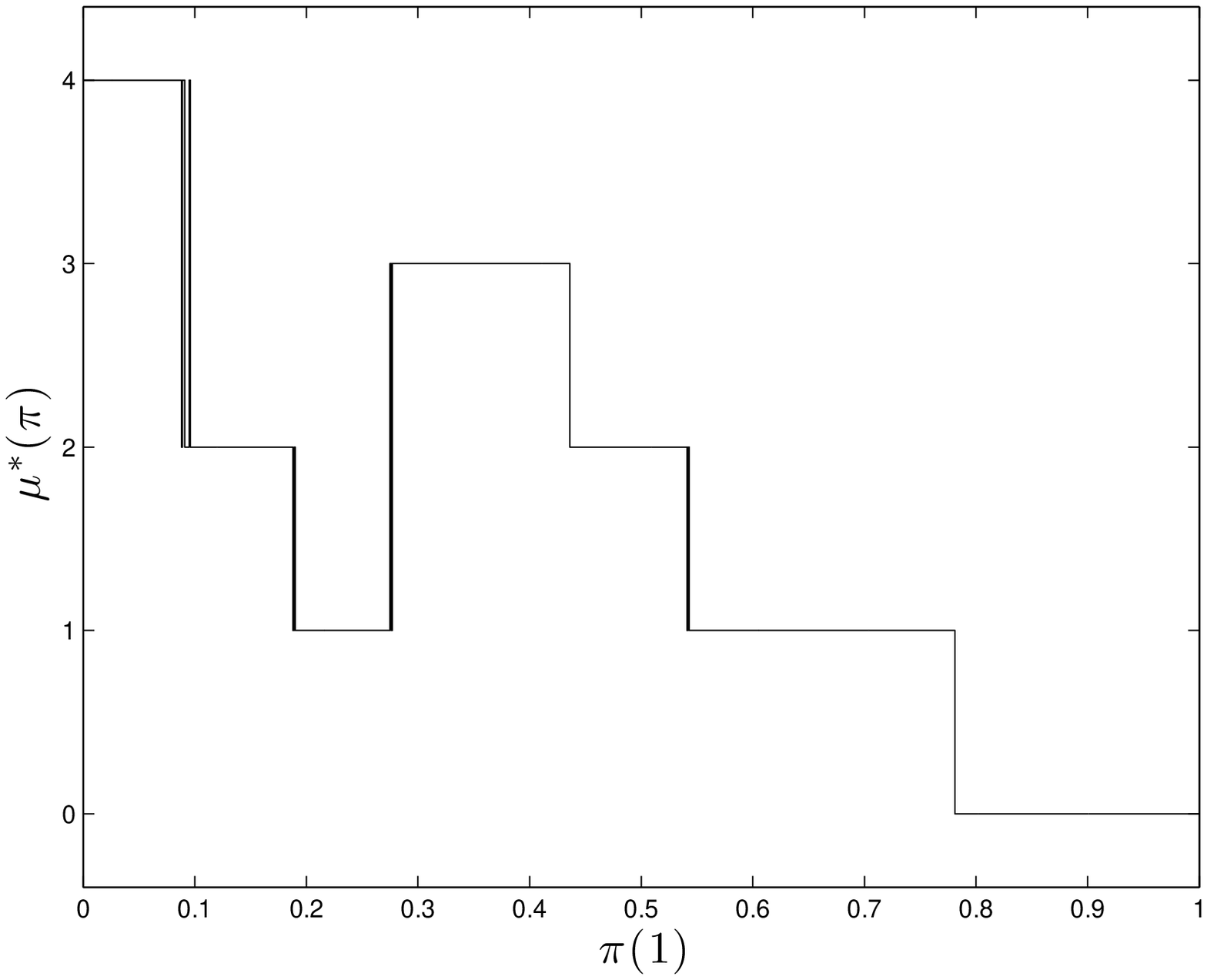,width=0.35\linewidth} }}
\caption{Optimal sampling strategy $\mu^*(\pi)$ for action space $u\in \{ 0 \text{ (announce change) }, D_1=1,D_2=3,\;D_3=5,\;D_4=10\}$ for a quickest-change detection
problem with geometric change time.  The noisy observations are from
a binary erasure channel and the parameters are specified in Example 1 of Sec.\ref{sec:numerical}.
 Figure
 \ref{fig:thm1}(a) depicts a monotone decreasing optimal strategy in posterior $\pi(1)$. Theorem \ref{thm:main1} gives sufficient conditions under which 
 the optimal sampling strategy $\mu^*(\pi)$ has this structure. The threshold
values  $\pi_1^*,\pi_2^*,\pi_3^*,\pi_4^*$ give a finite dimensional characterization of the optimal strategy.
Fig.\ref{fig:thm1}(b) gives an example where  the conditions of Theorem \ref{thm:main1} are violated
and the optimal strategy is no longer monotone in $\pi(1)$.}
\label{fig:thm1}
\end{figure}

{\em Example 2. Sensor Measurement Scheduling}: 
In  sensor and radar resource management problems, the sensor is a resource that needs
to be allocated amongst several targets \cite{Kri02,KD07}. Deploying a sensor to  look at a target consumes sensor resources. How should a sensor scheduler
 decide how often to look at a target in order to detect if the target has made a sudden maneuver (modelled by the Markov chain
 jumping to a target state)? In radar resource management
 \cite{KBGM12} this is called the revisit time problem.

\section{Formulation of Optimal Sampling Problem} \label{sec:qdbasic}

 Let $\time = 0,1,\ldots$ denote discrete time and 
$x_\time$  denote 
a Markov chain on
the finite state space 
\beq  \{e_1,\ldots,e_X\} \text{  where $e_i$ is the $X$-dimensional unit vector with 1 in the $i$-th position}. \label{eq:ei}\eeq
 Here  state `1'  (corresponding to $e_1$) is labelled as the ``target state''. 
  Denote
\beq 
\X = \{1,2,\ldots,X\}
. \eeq
Denote $X\times X$ 
transition probability matrix $\tp$  and the $X\times 1$ initial distribution $\pi_0$ where
\beq \tp=(\tp_{ij},i,j\in \X), \; \tp_{ij} = P(x_{\time+1}=e_j|x_\time=e_i), \quad \pi_0 = (\pi_0(i), i \in \X),\;  \pi_0(i) = 
P(x_0 = e_i).\label{eq:init} \eeq

\subsection{Measurement Sampling Protocol} \label{sec:qdmodel}

  Let  $\epoch_1,\ldots,\epoch_{k-1}$ denote previous discrete time instants at which measurement
samples  were taken. 
Let $\epoch_k$ denote the current time-instant at which a measurement is taken. The measurement sampling protocol proceeds according to the  following
steps:

 {\em Step 1. Observation}: A noisy measurement $y_k \in \Y$ at time $\epoch_k$  of the Markov chain is  obtained with
conditional probability distribution
 \beq \label{eq:obs}
 P(y_{k} \leq \bar{y}|x_{\epoch_k} = e_x) = \sum_{y\leq \bar{y}} B_{xy},\quad x \in  \X \eeq
Here $\sum_y$ denotes integration with respect to the Lebesgue measure
(in which case $\Y\subset \reals$ and $B_{xy}$ is the conditional probability density function)
or counting measure (in which case $\Y$ is a subset of the integers and $B_{xy}$ is the conditional probability
mass function $B_{xy}= P(y_{k}=y|x_{\epoch_k}=e_x)$).

{\em  Step 2. Sequential Decision Making}:  
Denote the filtration generated by  measurements and past decisions (denoted $u_1,\ldots,u_{k-1}$) as
  \beq
  \label{eq:sigf} \F_k = \text{ $\sigma$-algebra generated by } (y_{1},\ldots,y_{k},u_1,\ldots,u_{{k-1}}). \eeq
 At time $\epoch_k$, a $\F_k$ measurable decision $u_k \in \U $ is taken
where action
\begin{align}  \label{eq:actionstrategya}
 u_{k} &= \mu(\F_k)  \in \U = \Delayseti \\
  \text{ and } & u_k = l \text{ denotes:  obtain next measurement after } D_l \text{ time points}, l \in \delayseti . \nonumber
 \end{align}
In (\ref{eq:actionstrategya}),  the strategy $\mu$  belongs to the class of stationary decision
strategies denoted $\Mu$. Also, $D_1,\ldots,D_L$ are distinct positive integers that denote the set of possible sampling time intervals.
Thus the  decision $u_{k}$ specifies the next time $\epoch_{k+1}$ to make a measurement as follows:
\beq
\epoch_{k+1} = \epoch_k  + D_{u_{k}} , \quad u_k \in \delayseti.\eeq

{\em Step 3. Costs}:
Associated with the decision $u_{k} \in \U$, a  cost 
$c(x_t,u_{k})$ is incurred by the decision-maker at each time $\time \in [\epoch_k,\ldots,\epoch_{k+1}-1]$ until the next measurement is taken at time $\epoch_k$. Also a non-negative measurement sampling cost    $\mc(x_{\epoch_k},u_{k})$ is incurred. 

{\em Step 4}: If $u_{k}  = 1$, the problem terminates, else 
set $k$ to $k+1$ and go to Step 1.  \qed \\

\subsubsection*{Belief State Formulation} It is convenient to re-express Step 2 of the above protocol
in terms of the belief state.  It is well known from elementary stochastic control \cite{Lov87} that the belief state
(posterior) constitutes a sufficient statistic for $\F_k$ in (\ref{eq:sigf}).
Denote the belief state as $\pi_k = \E\{x_{\epoch_k}|\F_k\}$. Since the state space (\ref{eq:ei}) comprises of unit indicator vectors, conditional
probabilities and conditional expectations coincide. So
\beq   \pi_k = (\pi_k(i), \, i \in \X) \text{ where } \pi_k(i) = P(x_{\epoch_k} = e_i|y_{1},\ldots,y_{k},u_{1},\ldots,u_{{k-1}}), \text{ initialized by } \pi_0.
\eeq
It is easily proved that the belief state
 is updated via the Bayesian (Hidden Markov Model) filter 
\begin{align}  \label{eq:hmm} \pi_k &= \Tp(\pi_{k-1},y_k,u_{k-1}),  \text{ where }
\Tp(\pi,y,u) = 
\frac{B_y ({\tp^{\p}})^{D_u}\pi}{\sigp(\pi,y,u)},
\; \sigp(\pi,y, u) = \ones^\p B_y  ({\tp^\p})^{D_u} \pi 
\\
B_y &= \text{diag}(P(y|x),x\in \X).  \nonumber
 \end{align}
 Here $\sigma(\pi,y,u)$ is the normalization measure of the Bayesian update with $\sum_y \sigma(\pi,y,u) = 1$. Also
  $\ones$ denotes the $X$ dimensional vector of ones.
Note that  $\pi$ in (\ref{eq:hmm})  is an $X$-dimensional probability vector. It belongs to the
 $X-1$ dimensional unit-simplex denoted as
\begin{align}
\I \ole \left\{\pi \in \reals^{X}: \ones^{\p} \pi = 1,
\quad 
0 \leq \pi(i) \leq 1 \text{ for all } i \in \X \right\} \label{eq:Pi}
\end{align}
For example, $\Pi(2)$ is a one dimensional simplex (unit line segment),
$\Pi(3)$ is a two-dimensional  simplex (equilateral triangle); $\Pi(4)$
is a tetrahedron, etc.
Note that the unit vector states  $e_1,e_2,\ldots, e_X$ defined in (\ref{eq:ei}) of the Markov chain $x$ are the vertices of $\I$.

Step 2 in the above protocol expressed in terms of the belief state reads: At decision time $\epoch_k$
\begin{itemize}
\item {\em Step 2(a)}. Update belief state $\pi_k$ according to Bayesian filter  (\ref{eq:hmm}) 
\item {\em Step 2(b)}. Make decision $u_k \in \U$ using stationary strategy $\mu$ as (see (\ref{eq:actionstrategya}))
\beq \label{eq:actionstrategy}
 u_{k} = \mu(\pi_k)  \in \U = \Delayseti. \eeq
\end{itemize}

\subsection{Sequential Decision-maker's Objective and Stochastic Dynamic Programming} Given the above protocol with measurement-sampling strategy $\mu$ in (\ref{eq:actionstrategy}), we now
define the objective of the sequential decision maker.
Let $(\Omega,\mathcal{F})$ be the underlying measurable space where $\Omega = (\X \times \U \times \Y)^\infty$ is the product space, which is endowed with the product topology and $\mathcal{F}$  is the corresponding product sigma-algebra. For any $\pi_0\in \I$,  and strategy $\mu \in \Mu$,
 there exists a (unique) probability measure $\P^\mu_{\pi_0}$ on  $(\Omega, \mathcal{F})$, 
   see 
 \cite{HL96} for details. Let $\Epzero$  denote the expectation with respect to the measure  $\P^\mu_{\pi_0}$.
 
Define the  $\{\F_k,k\geq 1\}$ measurable stopping time $\kstar$
as
\beq
\kstar = \{ \inf k:  u_k = 0 \text{ (announce target state and stop) } \} .   \label{eq:kstar}\eeq
That is, $\kstar$ is the time at which the decision maker declares the target state has been reached and the problem terminates.
 For each initial distribution $\pi_0 \in \I$, and strategy
 $\mu$, the decision maker's global objective function is
 \beq
 J_\mu(\pizero) = \Ep\left\{\sum_{k=1}^{\kstar-1}  
\left[\mc(x_{\epoch_k},u_k) + 
 \sum_{\time = \epoch_k}^{\epoch_{k+1}-1}
c(x_t,u_k) \right]
+ c(x_{{\epoch_{\kstar}}}, u_{{\kstar}})
 \right\} \label{eq:rawcost}
 \eeq
  Recall that $c(x,u)$ and measurement sampling cost $\mc(x,u)$ are defined in Step 3 of the protocol.
 Using the smoothing property of conditional expectations, (\ref{eq:rawcost}) can be expressed in terms of the belief state $\pi$ as
\begin{align}
 J_\mu(\pizero) &= \Ep\left\{\sum_{k=1}^{\kstar-1}  C(\pi_{k},u_{k}) 
+C(\pi_{{{\kstar}}}, u_{{\kstar}}=0)
 \right\}   \label{eq:cost}\\
\text{ where }  C(\pi,u) &=  C_u ^\p \pi  \text{ for } u \in \U, \; \text{ and the $X$-dimensional cost vector $C_u$ is } \nonumber \\
C_0 &=  \begin{bmatrix}c(e_1,0),\ldots,c(e_X,0)\end{bmatrix}^\p , \nonumber\\  C_u &=
\mc_u + (I+ \tp + \cdots + \tp^{D_u-1}) c_u \text{ for } u \in \delayseti,\nonumber\\
 c_u &= \begin{bmatrix}c(e_1,u),\ldots,c(e_X,u)\end{bmatrix}^\p, \quad
\mc_u = \begin{bmatrix} \mc(e_1,u),\ldots, \mc(e_X,u) \end{bmatrix}^\p .
\nonumber
 \end{align}

The decision-maker aims to determine the  optimal strategy $\mu^* \in \Mu$ to minimize (\ref{eq:cost}),
i.e., 
\beq \label{eq:mustar}
 J_{\mu^*}(\pizero) = \inf_{\mu \in \Mu} J_\mu(\pizero) .\eeq The existence of an optimal stationary strategy $\mu^*$ follows from \cite[Prop.1.3, Chapter 3]{Ber00b}.


Considering  the  global objective  (\ref{eq:cost}),
the optimal stationary strategy $\mu^*: \I \rightarrow \U$ and associated optimal objective $J_{\mu^*}(\pi)$ are the solution of the following
 ``Bellman's stochastic dynamic programming  equation'' 
\begin{align} \label{eq:dp_alg}
\mu^*(\pi)&= \arg\min_{u \in \U} Q(\pi,u) , \;J_{\mu^*}(\pi) = V(\pi) = \min_{u \in \U} Q(\pi,u),\\
 \text{ where }  Q(\pi,u) &=  C(\pi,u)
+ \sum_{y \in \Y}  V\left( \Tp(\pi ,y,u) \right) \sigp(\pi,y,u),\; u =1 ,\ldots,L, \quad
Q(\pi,0) =  C(\pi,0).
  \nonumber
\end{align}
Recall $T(\pi,y,u)$ and  $\sigma(\pi,y,u)$ were defined in (\ref{eq:hmm}).
The above formulation is a generalization of a  partially observed Markov decision process (POMDP), since  POMDPs assume
finite observations spaces $\Y$ while in our formulation $\Y$ can be discrete or continuous (see  (\ref{eq:obs})).

Define the   set of belief states where it is optimal to apply action $u=0$ as \begin{multline} 
\Stop =  \{\pi \in \I : \mu^*(\pi) = 0\} = \{\pi \in \I:Q(\pi,0) \leq Q(\pi,u),
 \; u \in \delayseti \} 
 \label{eq:stopset}
 \end{multline}
$\Stop$ is called the {\em stopping set} since it is the set of belief states to ``declare  target state and stop''. 

Since the belief state space $\I$ is an uncountable, Bellman's equation (\ref{eq:dp_alg}) does not translate directly  into 
numerical algorithms. However,  in subsequent sections, we exploit the
structure of Bellman's equation
 to prove various structural results about the optimal strategy $\mu^*$ using lattice programming and stochastic dominance tools.

\subsection{Example: Quickest Change Detection with Measurement Control} \label{sec:exqd}
We now formulate the quickest detection problem with optimal sampling -- this serves as a useful example to illustrate
 the above general model.
Before proceeding, it is important to recall  that in our  model, decisions (whether to stop, or continue and take next observation sample
after $D_l$ time points) are made at times $\epoch_1,\epoch_2, \ldots$. In contrast,  the state of the Markov chain  (which models the change
we want to detect) can change
at any time $t$. We need to construct the  delay penalty and false alarm penalties carefully
to   take this into account. 
\\
1.  {\em Phase-Distributed (PH) Change time}:  In quickest detection, the target state (which we label as state 1 by convention) is an absorbing state.
States $2,\ldots,X$ (corresponding to unit vectors $e_2,\ldots,e_X$) are now  fictitious states that form  a single composite state that 
the Markov chain $x_t$ resides in before jumping into the target absorbing state.
So the transition matrix (\ref{eq:init})  is
\beq \label{eq:phmatrix}
\tp = \begin{bmatrix}  1 & 0 \\ \underline{\tp}_{(X-1)\times 1} & \bar{\tp}_{(X-1)\times (X-1)} \end{bmatrix}.
\eeq
The ``change time" $\changetime$ denotes the time at which $x_\time$ enters the absorbing state 1,
i.e., \beq \changetime = \inf\{\time > 0: x_\time = 1\} . \label{eq:tau}\eeq
 Of course, in the special case when $x$ is a 2-state Markov chain (i.e., $X=2$),
 the change time $\changetime$ in (\ref{eq:tau})  is geometrically distributed. \\
For the multi-state case, to ensure that $\changetime$ is finite, assume states $2,3,\ldots X$ are transient.
This  is equivalent to $\bar{\tp}$  in (\ref{eq:phmatrix}) satisfying  $\sum_{n=1}^\infty \bar{\tp}^{n}_{ii} < \infty$ for $i=1,\ldots,X-1$ (where $\bar{\tp}^{n}_{ii}$ denotes the $(i,i)$ element of the $n$-th power
of matrix $\bar{\tp}$).
With the  transition probabilities (\ref{eq:phmatrix}), the distribution of the change time $\changetime$ is given by the PH-distribution
\beq \label{eq:nu}
 P(\changetime = 0 )= \pi_0(1), \quad P(\changetime= \time)  = \bar{\pi}_0^\p \bar{\tp}^{\time-1} \underline{\tp}, \quad \time\geq 1 \eeq
 where $\bar{\pi}_0 = [\pi_0(2),\ldots,\pi_0(X)]^\p$.
By  choosing  $(\pi_0,\tp)$ 
and state space dimension $X$,
 one can approximate any given  change-time distribution on $[0, \infty)$ by PH-distribution (\ref{eq:nu}); see 
\cite[pp.240-243]{Neu89}.  Indeed, PH-distributions form a dense subset for the set of all distributions.
\\
2. {\em Observations}: Since   states $2, 3,.\ldots, X$ are fictitious states that shape the PH-distributed change time (\ref{eq:nu}),
they are indistinguishable in terms of the observation $y$.
That is, $B_{2y} =B_{3y} = \cdots = B_{Xy}$ for all $y\in \Y$.
\\
3. {\em  Costs}: Associated with the quickest detection problem are the following costs. \\
{\em (i) False Alarm}:
Let $\kstar$ denote the time $\epoch_k$ at which  decision $u_{k}=0$ (stop and announce target state) is chosen, so that  the problem terminates.
If the decision to stop is made before the Markov chain reaches the target state 1, i.e., $\kstar < 
\changetime$, then a false
alarm penalty $\f$  is paid.  So the false alarm penalty is $\f \sum_{i\neq 1} I(x_{\epoch_k}=e_i,u_k=1)$ where $\f$ is a user defined non-negative constant.
The expected false alarm penalty based on the accumulated history is
\beq  \sum_{i \neq 1} \f \E\{I(x_{\epoch_k}=e_i,u_{k}=1)|\F_k\} = \f (\ones - e_1)^\p \pi_k  I(u_k=1).\eeq
  Recall $\ones$ denotes the $X$-dimensional vector of ones.\\
(ii) {\em Delay cost of continuing}: Suppose decision $u_{k} \in \delayseti$ is taken at time $\epoch_k$. So the next sampling
time  is $\epoch_{k+1} = \epoch_k + D_{u_k}$.
Then for any time $\time \in  [\epoch_k, \epoch_{k+1}-1]$,  the event $ \{x_{\time} = e_1, u_{k}\}$ signifies that a change
has occurred but not been announced by the decision maker. Since the decision maker can make the next decision (to stop
or continue) at $\epoch_{k+1}$, the
 delay cost incurred  in the time  interval  $[\epoch_k, \epoch_{k+1}-1]$ is
  $d \sum_{t=\epoch_k}^{\epoch_{k+1}-1} I(x_{t} = e_1, u_{k} )$ 
  where   $d$ is a non-negative constant.
The expected delay cost in interval  $[\epoch_k, \epoch_{k+1}-1] = [\epoch_k, \epoch_k+D_{u_k}-1] $ is
\beq  d  \sum_{t=\epoch_k}^{\epoch_{k+1}-1} \E\{I(x_{t}= e_1,u_k )|\F_{k}\} 
= d e_1^\p (I + \tp + \cdots + \tp^{D_{u_k}-1})^\p \pi_k, \quad u_k \in \delayseti. \eeq
%
(iii) {\em Measurement Sampling Cost}: Suppose decision $u_{k} \in \delayseti$ is taken at time $\epoch_k$.
As in (\ref{eq:cost}) let   $\mc_{u_{k}} = (\mc(x_{\epoch_k}=e_i,u_{k}), i\in \X)$ denote the non-negative measurement  cost  vector for choosing to take a measurement. 
For convenience, assume the measurement cost when choosing $u=0$ (stop) is zero.
Next, since in quickest detection, states $2,\ldots,X$ are fictitious states that are indistinguishable in terms of cost, choose $\mc(e_2,u) =\ldots = \mc(e_X,u)$.\\  Examples of measurement sampling costs are:\\
(a) $\mc(e_i,u)$ is independent of state $i$ and action $u$.  This simple choice of a constant measurement cost at each time, still
results in non-trivial global costs for the decision maker since
this   cost is incurred  each time a measurement is made -- so
choosing a decision $u$ with smaller sampling delay will result in 
 more measurements until the final decision to stop, thereby incurring a higher
total measurement cost for the global decision maker.
\\
(b)  $\mc(e_i,u)$ is decreasing in $u$ for $u\neq 0$.  
Choosing $\mc(e_i,u)$ to decrease in $u$  penalizes  choosing small sampling intervals  even more than  a constant cost.

{\bf Summary and  Kolmogorov--Shiryayev criterion}: To summarize, the costs  $C(\pi,u)$  for quickest detection with optimal sampling and PH-distributed change time are
\begin{align}
C(\pi,0) &= \f (\ones - e_1)^\p \pi , \quad C(\pi,u) = c_u^\p \pi + \mc_u^\p \pi , \text{ for } u \in \delayseti , \nonumber\\
\text{ where } &
c_u = d (I+\tp+ \cdots + \tp^{D_u-1}) e_1  .\label{eq:qdcosts} \end{align}

For  constant
measurement cost $\mc(x,u) = m$, $u\in \delayseti$, 
the quickest detection optimal sampling objective
 (\ref{eq:cost}) with costs (\ref{eq:qdcosts})  can be expressed as  
 \beq J_\mu(\pizero) =   d \Ep\{(\kstar - \changetime)^+\} + \f \P^\mu_{\pi_0}(\kstar < \changetime)  +
 m \sum_{u=1}^L \sum_{k:\epoch_k \leq \kstar} \P^\mu_{\pi_0}(u_k = u)
\label{eq:ks}
\eeq
where the PH-distributed change time
 $\changetime$  and $\kstar$ are
defined in (\ref{eq:tau}),  (\ref{eq:kstar}).
For the special case
$\U=\{0 \text{ (stop)} ,\D_1=1\}$, measurement cost $\mc_u = 0$,  geometrically
distributed $\changetime$ (so $X=2$), then (\ref{eq:ks}) 
becomes the  Kolmogorov--Shiryayev
criterion for detection of disorder \cite{Shi63}.

\section{Structural Results for Optimal Sampling Policy $\mu^*(\pi)$ for 2-state case} \label{sec:geom}
This section analyzes the structure of the optimal sampling strategy $\mu^*(\pi)$ (solution of Bellman's equation (\ref{eq:dp_alg})) for two-state Markov chains ($X=2$).
Recall that two-state Markov chains model geometric distributed change times in quickest detection problems.

We list the following assumptions that will be used in this section.

\begin{compactitem}
\item[{\bf (A1)}] (i) The costs $C(e_i,u)$  in (\ref{eq:cost})  are increasing with $i \in \X$ for each $u \in \U$.\\
(ii) The target state $e_1$ belongs to the stopping set $\Stop$ defined in (\ref{eq:stopset}).
\item[{\bf (A2)}] The transition matrix $\tp$ is totally positive of order 2 (TP2). That is, all second order minors are non-negative.
\item[{\bf (A3)}] The observation matrix $B$ is TP2.
\item[{\bf (A4)}] $C(e_i,u)$ is submodular for $u\in\delayseti$, that is $C(e_i,u+1)-C(e_i,u)$ is decreasing\footnote{Throughout this paper, we use the term ``decreasing" in the weak sense. That is ``decreasing" means
non-increasing. Similarly, the term ``increasing" means non-decreasing.}  in $i \in \X$.
\end{compactitem}

Consider the following assumption  where  $\tp^{D_u}\vert_{ij}$ denotes the $(i,j)$ element of matrix $\tp^{D_u}$:

\begin{itemize}
\item[{\bf (A5-(i))}]  For each $q\in \X$,  $\sum_{j\geq q} \tp^{D_u}\vert_{i j}$ is submodular. That is,
$\sum_{j\geq q} \tp^{D_{u+1}}\vert_{ij} - \tp^{D_u}\vert_{ij}$  is decreasing in $i \in \X$ for $u=1,2,\ldots,L-1$.
\item[{\bf (A5-(ii))}] $\tp^{D_u}\vert_{22}$ and $\tp^{D_u}\vert_{12}$ is decreasing in $u \in \delayseti$.
\end{itemize}

These assumptions are discussed below in Sec.\ref{sec:discussion} and hold for quickest detection problems.

\subsection{Optimality of Threshold Policy for Sequential Optimal Sampling} \label{sec:2state}
Note that for a 2-state Markov chain ($X=2$), the belief state space $\I$ is the one dimensional simplex $\pi(1)+\pi(2) = 1$. So it suffices to represent $\pi$ by its first element
 $\pi(1)$. 

\begin{framed}
\begin{theorem} \label{thm:main1}
Consider the optimal sampling  problem of Sec.\ref{sec:qdbasic} with state dimension $X=2$ and action space $\U$ (\ref{eq:actionstrategya}). Then
the optimal strategy $\mu^*(\pi)$ in (\ref{eq:dp_alg}) has the following structure:
\\
(i) The optimal stopping set $\Stop$ (\ref{eq:stopset}) is a convex subset of $\I$. Therefore under (A0), the stopping set is the interval
 $\Stop = (\pi_1^*,1]$ where
the threshold $\pi^* \in [0,1]$.\\
(ii) Under
assumptions (A1-A5),  the optimal sampling strategy  $\mu^*(\pi)$ in (\ref{eq:dp_alg}) is  decreasing with $\pi(1)$. Thus, there exist up to $L$ thresholds denoted $\pi_1^*, \ldots, \pi_L^*$ with
 $0 \leq \pi_L^* \leq \pi_{L-1}^* \leq \cdots \leq \pi_1^* \leq 1 $ such that the optimal strategy satisfies
 \beq \label{eq:thres}
 \mu^*(\pi) = \begin{cases}  L \text{ (sample after  $D_L$ time points) } & \text{ if } 0 \leq \pi(1) < \pi_L^* \\
 					{L-1} \text{ (sample after  $D_{L-1}$ time points) } & \text{ if }  \pi_L^* \leq \pi(1) < \pi_{L-1}^* \\
					\vdots & \vdots \\
					 1 \text{ (sample after  $D_{1}$ time points) } & \text{ if } \pi_{2}^* \leq \pi(1) < \pi_1^*  \\
					 0 \text{ (announce change )} & \text{ if } \pi_1^* \leq \pi(1) \leq 1 \end{cases}
 \eeq
 where the sampling delays are ordered as $D_1 < D_2 < \ldots < D_L$.
\end{theorem}
\end{framed}

The proof of Theorem \ref{thm:main1} is in Appendix \ref{app:main1}.
As an  example, consider   quickest detection with optimal sampling for geometric distributed change time.
 From (\ref{eq:phmatrix}),   the transition matrix is 
$ \tp = \begin{bmatrix} 1 & 0 \\ 1-\tp_{22} & \tp_{22} \end{bmatrix}$ and expected change time is $\E\{\changetime\} = \frac{1}{1-\tp_{22}}$
where $\changetime$ is defined in (\ref{eq:tau}).

\begin{framed}
\begin{theorem}\label{cor:qd}
Consider the  quickest detection problem with optimal sampling  and geometric-distributed change time formulated in Sec.\ref{sec:exqd} with costs
defined in (\ref{eq:qdcosts}).
 Assume the measurement cost $\mc(x,u)$ satisfies (A1) and (A4), e.g., the measurement cost is a constant.
Then if (A3)  holds, 
Theorem \ref{thm:main1} holds. So the optimal sampling strategy (\ref{eq:thres}) makes measurements  less 
frequently when away from the target state and more  frequently when
closer to the target state.   (Note, (A1)(ii),  (A2), (A5) hold automatically and no assumptions are required on the delay or stopping costs in
(\ref{eq:qdcosts})).
\end{theorem}
\end{framed}

There are two main conclusions regarding Theorem \ref{cor:qd}. First, for constant measurement cost, Theorem \ref{cor:qd} holds
without any assumptions for Gaussians, exponentials, and several other 
 classes
 of observation distributions that satisfy (A3).
Second, the optimal strategy 
 $\mu^*(\pi)$ is monotone  in posterior $\pi(1)$ and therefore has a finite dimensional characterization. To determine the optimal strategy, one only
 needs to determine (estimate) the values of the $L$ thresholds 
 $\pi_1^*, \ldots, \pi_L^*$. These can be estimated via a simulation-based stochastic optimization algorithm. We will  give bounds for these threshold values
 in Sec.\ref{sec:main}. Fig.\ref{fig:thm1} illustrates such a monotone policy.

A short word on the proof  presented in Appendix \ref{app:main1}. It  involves
 analyzing the structure of Bellman's equation (\ref{eq:dp_alg}). It will  be shown that
 $Q(\pi,u)$  in   (\ref{eq:dp_alg}) is a submodular function (defined in Appendix \ref{app:main1}) on  the partially ordered set $[\I,\gr]$ 
 which constitutes a lattice. Here $\gr$ denotes the monotone likelihood
 ratio stochastic order defined in Sec.\ref{sec:filter}.
 For $X=2$, $\I$ is  the unit interval [0,1]  and in this case $[\I,\gr]$ is a chain (totally ordered set) and $\gr$ is equivalent to first order 
 stochastic dominance. For $X\geq 2$ considered in the next section, a similar idea is used to bound the optimal policy on $[\I,\gr]$.

{\em Remark: Interpretation of \cite{BV12}}.  We comment here briefly on the recent paper \cite{BV12} which considers quickest  detection
with measurement control where
at each time the decision is made  whether to take a measurement
or not. This
 can be formulated as our optimal sampling problem by considering the action
space $\U =\Delayset$ with sampling interval $D_i = i$ and $L$ chosen sufficiently large.
In \cite{BV12}, a different action space is chosen, namely $\{ 0 \text{ (announce change)}, m \text{ (take measurement)}, \bar{m} \text{ (no measurement)} \}$.
With this action space, \cite{BV12} shows that the optimal strategy is not necessarily monotone in the posterior. $\pi$.  However, with the
action space $\U$ defined above, Theorem \ref{cor:qd} shows that the optimal strategy indeed is monotone.

We can interpret the non-monotone optimal strategy in \cite{BV12} as follows.  
Our action 1 (sample next point) corresponds to action $m$ in \cite{BV12}, action 2 (sample after 2 points) corresponds to $(\bar{m},m)$, action 3 (sample after 3 points) corresponds to $(\bar{m},\bar{m},m)$, etc.
Reading off the monotone optimal strategy $\mu^*(\pi) \in 
\{3, 2,1\}$ versus $\pi$ using the action
space of \cite{BV12}  yields strategy $\{\bar{m},\bar{m},m, \bar{m},m, m\}$ which
is non-monotone, due to presence of the action $\bar{m}$ sandwiched between two $m$'s.\footnote{\cite{BV12} also contains a very nice performance analysis of sub--optimal and nearly optimal strategies. This analysis may be applicable to our setup due to similarity of the models.}

\subsection{Discussion of Assumptions A1-A5}\label{sec:discussion}  To illustrate the assumptions of Theorem \ref{thm:main1}, we will now prove Theorem \ref{cor:qd}
by showing that assumptions  that (A1-A5)  hold.
 Recall  from (\ref{eq:phmatrix}) that for quickest detection with geometric change time, the transition matrix is 
\beq \tp = \begin{bmatrix} 1 & 0 \\ 1-\tp_{22} & \tp_{22} \end{bmatrix}.  \text{ So } \;
\tp^{D_u} = \begin{bmatrix} 1 & 0 \\ 1-\tp^{D_u}_{22} & \tp^{D_u}_{22} \end{bmatrix} . \label{eq:tpg}\eeq
\subsubsection{Assumption (A1)(i)} This
requires the elements of the cost vector to be increasing. However, in quickest detection, the instantaneous cost $c(e_i,u)$ for $u\geq 1$  defined in (\ref{eq:qdcosts}) is {\em decreasing} in $i$ if the measurement cost $m$
 is a constant. 
But all is not lost.
Remarkably, a  clever transformation can be applied to make a transformed version  of the cost  {\em increase} with $i$ and yet 
ensure (A4) holds and keep the optimal strategy unchanged! This transformation is crucial for proving Theorem~\ref{cor:qd}, particularly for constant measurement cost. (Assuming a measurement cost increasing in the state $i$ makes
the proof  easier but may be unrealistic in applications).
We define this transformation via the following theorem - it exploits the special structure of the quickest detection problem.

\begin{theorem} \label{lem:transform}
Consider the quickest time detection problem with costs defined in  (\ref{eq:qdcosts}).
Assume the sampling cost $m(e_i,u)$ is a constant. Define the transformed costs $\uC(\pi,u)$  as follows:
\begin{align}  \uC(\pi,0) &= C(\pi,0) - \alpha C(\pi,L),  \nonumber \\
\uC(\pi,u) &=  C(\pi,u) - \alpha C(\pi,L) + \alpha \sum_y C(T(\pi,y,u),L) \sigma(\pi,y,u), \quad u \in \delayseti 
\label{eq:transform}
\end{align}
for any constant $\alpha \in \reals$.
Then:\\
(i) Bellman's equation (\ref{eq:dp_alg}) applied to optimize the    global objective (\ref{eq:cost}) with transformed costs $\uC(\pi,u)$
yields the same optimal strategy  $\mu^*(\pi)$ as the global objective with  original costs $C(\pi,u)$.
\\
(ii) Choosing $\alpha = 1/(1-\tp_{22}^{D_u})$  implies $\uC(e_i,u)$ satisfies (A1) and (A4).
 \qed
\end{theorem}

 Theorem \ref{lem:transform} is proved in Appendix \ref{app:transform}.
It 
 asserts that the transformed costs $\uC(e_i,u)$  satisfies (A1) and (A4) even for constant measurement cost. 
Therefore Theorem \ref{thm:main1} holds and the optimal strategy for the transformed costs is monotone in the posterior distribution.
Note that  Theorem \ref{lem:transform} also says that  the optimal
strategy $\mu^*(\pi)$ is unchanged by this transformation. Thus  Theorem \ref{thm:main1}  holds for the original quickest detection costs,
thereby proving Theorem \ref{cor:qd}.

 {\em Assumption (A1)(ii)} is natural for the stopping problem to be well defined. It says that if it was known with certainty that the target state $e_1$
has been reached, then it is optimal to  stop. For quickest time detection it holds trivially since $C(e_1,0) \leq  C(\pi,u)$ for  $u\in \{1,\ldots,L\}, \pi \in \I$.

\subsubsection{Assumption
(A2)}  From  the structure of transition matrix $\tp$ in (\ref{eq:tpg}),  clearly (A2) holds automatically for the quickest detection problem.  For numerous examples of TP2 transition matrices, see \cite{KR80}. Also,  $\tp$ does not need to have an
absorbing state for Theorem \ref{thm:main1} to hold.

\subsubsection{Assumption (A3)} Numerous continuous and discrete  noise distributions satisfy the TP2 property, see \cite{KR80}. Examples include Gaussians, Exponential, Binomial, Poisson, etc.
Examples of discrete observation distribution  satisfying (A3) include
binary erasure channels
-- see   Sec.\ref{sec:numerical}.  A binary symmetric channel with error probability less than 0.5 also satisfies (A3).
 
\subsubsection{Assumption (A4)} In general Theorem \ref{thm:main1} requires the costs $C(e_i,u)$ to be submodular. 
However, for the special case of quickest detection with optimal sampling, from Theorem \ref{lem:transform} shows that only
the measurement cost $\mc(e_i,u)$ needs to be submodular, i.e., $m(e_i,u+1)-m(e_i,u)$ is decreasing in $i$.  This holds trivially  if the measurement cost
is independent of the state.

\subsubsection{Assumption
(A5)} This is a submodularity condition on the transition matrix.  Since from (\ref{eq:tpg}) $\tp^{D_u}\vert_{21} = 0$ and $\tp^{D_u}\vert_{22} =  \tp_{22}^{D_u}$, clearly (A5) holds automatically for the quickest detection problem with optimal sampling.

\section{Myopic  Bounds to Optimal  Strategy for multi-state Markov Chain} \label{sec:main}
This section considers the optimal sampling problem for multi-state Markov chains ($X\geq 2$).
Recall that multi-state Markov chains can model PH-distributed change times (\ref{eq:phmatrix}) in quickest detection problems.
 Theorems \ref{thm:myopic} and \ref{cor:qd2} are the main results of this section.  They characterize the structure of
optimal strategy $\mu^*(\pi)$ which is the solution of Bellman's equation (\ref{eq:dp_alg}).

 Define the following ordering of two arbitrary transition matrices $\tpone$ and $\tptwo$: 
\beq \tpone\succeq \tptwo \text{ if } \tpone_{ij} \tptwo_{m,j+1} \leq \tptwo_{ij} \tpone_{m,j+1}, \quad i,j+1,m \in \X. \label{eq:mor}\eeq

The following are the main assumptions used in this section:
\begin{compactitem}
\item[{\bf (A6)}] The transition matrix satisfies $\tp \succeq \tp^{2}$ where the ordering $\succeq$ is defined in (\ref{eq:mor}).
\item[{\bf (A7)}] There exist a positive constant $\alpha$ satisfying \\ (i) $\alpha \geq f/(d(1-\tp_{21}))$ \\
(ii) $\sum_{l=1}^{D_u-1} \tp^l|_{i1} + \alpha (\tp_{i1} - \tp^{D_u+1}|_{i1}) $ decreasing in $i=2,\ldots,X$.
\end{compactitem}

(A6) and (A7) are discussed at the end of this section.
(A6) and (A7) hold trivially for the two state Markov  chain case ($X=2$) with absorbing state when $\tp$ is of the form (\ref{eq:tpg}). Examples for $X\geq 3$ are given in Sec.\ref{sec:numerical}.

\subsection{Myopic Lower Bound to Optimal Policy}
 For  multi-state Markov chain observed in noise, determining sufficient conditions for the
optimal strategy to have a  monotone structure
 is an open problem.  
In this section we show that the optimal sampling strategy $\mu^*(\pi)$ is lower bounded by a  myopic strategy.
Define the  myopic strategy  ${\umu}(\pi)$  and myopic stopping set $\uStop $ by
\begin{align} {\umu}(\pi) &= \arg\min_{u \in \U} C(\pi,u)    \nonumber\\
 \label{eq:ustop}
\uStop =   \{\pi \in \I : \umu^*(\pi) = 0\} &=  \{\pi \in \I:  C(\pi,0) <  C(\pi,u), \; u \in \delayseti \} \end{align}
So $\uStop$ is the set of belief states for which  the myopic strategy declares stop.
 
The following is the main result of this section. The proof is in Appendix \ref{app:myopic}.

 \begin{framed}
\begin{theorem} \label{thm:myopic}
Consider the sequential sampling  problem of Sec.\ref{sec:qdbasic} with optimal strategy specified by (\ref{eq:dp_alg}). Then
\begin{enumerate}
\item The stopping set $\Stop$ defined in (\ref{eq:stopset}) is a convex subset of the belief state space $\I$.
\item $\uStop \subset \Stop$ where $\uStop$ is the myopic stopping set defined in (\ref{eq:ustop}).
\item Under (A1), (A2), (A3), (A6),
the myopic strategy 
 ${\umu}(\pi)$ defined in (\ref{eq:ustop}) forms  a lower   bound for the optimal
strategy $\mu^*(\pi)$, i.e.,  $\mu^*(\pi) \geq {\umu}(\pi)$ for
all $\pi \in \I - \Stop$. 
\item If (A4) holds, then the myopic strategy is ${\umu}(\pi)$ is increasing with $\pi$ (with respect
to the monotone likelihood ratio (MLR) stochastic order to be defined in Sec.\ref{sec:filter}).  
\end{enumerate} 
\end{theorem} \end{framed}

The above theorem says  that the myopic strategy  ${\umu}(\pi)$  comprising of increasing step functions
lower bounds the optimal strategy $\mu^*(\pi)$.
The myopic strategy specified by (\ref{eq:ustop}) is computed trivially on the 
 simplex $\I$.
Therefore, the above theorem gives a useful lower bound ${\umu}(\pi)$ for the optimal strategy $\mu^*(\pi)$ (which
is intractable to compute). 
  Also since ${\umu}(\pi)$ is sub-optimal, it incurs a higher cost compared to the optimal strategy. This cost associated
  with ${\umu}(\pi)$ can be evaluated by simulation and forms an upper bound to the optimal achievable cost.

\subsection{Quickest Detection with Optimal Sampling for PH-Distributed Change Time}
We now use 
Theorem \ref{thm:myopic}  to construct a myopic strategy  that upper bounds  the optimal  strategy for  quickest detection with sampling for  PH-distributed change time. 
 However, (A1) does not hold for the quickest detection costs $C(\pi,u)$ in (\ref{eq:qdcosts})
To proceed, it is convenient to define the following transformed cost $\Cb(\pi,u)$ and myopic strategy  $\bmu(\pi) $:
\begin{align} \label{eq:bmu}
\Cb(\pi,0) &= C(\pi,0) + \alpha  d e_1^\p \tp ^\p \pi, \\
\Cb(\pi,u) &= C(\pi,u) + \alpha  d e_1^\p \tp ^\p \pi -  \alpha  d e_1^\p {\tp^\p}^{D_u+1}  \pi,  \quad u \in \delayseti, \nonumber\\
\quad  \bmu(\pi) &= \arg\min_{u \in \delayseti} \Cb(\pi,u) . \nonumber \end{align}

It will be shown in the proof of the theorem below, that the optimal strategy for global objective (\ref{eq:cost})
with these transformed costs $\Cb(\pi,u)$ remains unchanged and is still $\mu^*(\pi)$.


The main result regarding quickest detection with optimal sampling for PH-distributed change times is as follows. The proof is in Appendix \ref{app:qd2}.
\begin{framed}
\begin{theorem} 
\label{cor:qd2}
Consider the quickest detection  optimal sampling  problem for PH-distributed change time ($X\geq 2$) defined in Sec.\ref{sec:exqd} with
costs in (\ref{eq:qdcosts}) and  transformed costs (\ref{eq:bmu}). Then
\begin{compactenum}
\item The optimal stopping set $\Stop$  (\ref{eq:stopset}) is a convex subset of the belief state space $\I$ and contains $e_1$. 
\item   The optimal stopping set is lower  bounded
by the  myopic stopping set $\uStop$ in (\ref{eq:ustop}), i.e., 
$\uStop  \subset \Stop$.
\item Under (A2), (A3), (A6), (A7),  for $\pi \in \I - \uStop$, the myopic strategy $\bmu(\pi) $ in (\ref{eq:bmu})  upper bounds the optimal strategy, i.e., $\bmu(\pi) \geq \mu^*(\pi) \text{  for all } \pi \in \I- \uStop$. Moreover, $\bmu(\pi) $ is increasing  in $\pi$ with respect to the MLR order.
\item  For the case of geometrically distributed change time ($X=2$), 
 (A2) and (A6) hold automatically. So if (A3) holds and $\alpha$ is chosen according to (A7)(i), 
 then  $\bmu(\pi) \geq \mu^*(\pi)$ for all $\pi \in \I$. Also
  $\bmu(\pi)$
 is decreasing in $\pi(1)$. \end{compactenum}
 \end{theorem}
\end{framed}

{\em Discussion}: Theorem \ref{cor:qd2} gives a lot of analytical mileage in terms of bounding the optimal strategy.
 Statements (1)  characterizes  the convexity of the stopping set and  Statement (2) lower  bounds 
$\Stop$. 
Statement (3)  asserts that the optimal sampling strategy $\mu^*$ can be upper bounded for PH-distributed change times by the myopic
strategy $\bmu$.

 Statement (4) shows that for geometrically distributed change times,
the bounds on $\mu^*$ and $\Stop$ apply without requiring any assumptions apart from that on the observation distribution (A3).
In particular, Statement (4) together with Theorem \ref{cor:qd}  say that the optimal strategy $\mu^*(\pi)$ is monotone in $\pi(1)$, and upper bounds
for the threshold values $\pi_1^*, \pi_2^*,\ldots, \pi_L^*$ can be constructed using the myopic  strategy $\bmu(\pi)$. Therefore these upper bounds can be used to initialize
a stochastic optimization algorithm to estimate the thresholds of the optimal monotone policy.

{\em Discussion of Assumptions (A6) and (A7)}:
(A6) is a  sufficient to preserve monotone likelihood ratio (MLR) dominance of the belief state with a one-step predictor.
MLR stochastic order $\gr$ is defined
in Sec.\ref{sec:filter}. (A6) ensures that if two belief states satisfy $\pi_1\gr \pi_2$, then the one-step ahead
Bayesian predictor satisfies $\tp^\p \pi_1 \gr \tp^\p \pi_2$. Theorem \ref{thm:filter} in Sec.\ref{sec:filter} analyzes this
and other stochastic
dominance properties of Bayesian filters and predictors that are crucial to prove the results of this paper.

(A7) is sufficient  for the transformed cost $\Cb(\pi,u)$ defined in (\ref{eq:bmu}) to satisfy the following:
$-\Cb(\pi,u)$ satisfies (A1) and $\Cb(\pi,u)$ to satisfies (A4). The fact that $\Cb(\pi,u)$ is decreasing
(since its negative satisfies (A1)), gives an upper bound by a proof completely analogous to Theorem~\ref{thm:myopic}. 

\section{Performance and Sensitivity of Optimal Strategy}
In previous sections, we have presented structural results on monotone  optimal {\em strategies}. In comparison, this section focuses on
{\em achievable costs} attained by the optimal strategy.
This section presents  two results. First, we give bounds on the achievable performance of the optimal strategies by the decision maker.
This is done by introducing a partial ordering of the transition and observation probabilities -- the larger these parameters with respect to this order, the larger the 
optimal cost incurred. Thus we can compare models and bound the achievable performance of a computationally intractable problem.
Second, we give explicit bounds on the sensitivity of the total sampling cost with respect
to sampling model -- this bound can be expressed in terms of the Kullback-Leibler divergence. Such a robustness result
is useful since even if a model violates
the assumptions of the previous section, as long as the model is sufficiently close to a model that satisfies the conditions, then the optimal
strategy is close to a monotone strategy.

\subsection{How does total cost of the optimal sampling strategy depend on state dynamics?} \label{sec:order}

Consider the optimal sampling problem formulated in Sec.\ref{sec:qdbasic}. 
How does the  optimal expected  cost $J_{\mu^*}$ defined in (\ref{eq:mustar}) vary
  with
 transition matrix $\tp$ and observation matrix $B$? In particular, is it possible to devise an ordering for transition
 matrices and observation distributions
 such that the larger they are, the smaller the optimal sampling cost? Such a result would allow us to compare the optimal
 performance of different sampling models, even though computing these is intractable. Moreover, characterizing how the optimal
 achievable cost varies with transition matrix and observation distribution is useful.
  Recall in quickest detection the transition matrix specifies the change time distribution.
 In sensor scheduling applications the transition matrix specifies the mobility of the state. The observation matrix specifies
 the noise distribution.
 
 Consider two distinct models $\model= (\tp,B)$ and $\bmodel = (\btp,\bB)$ of the optimal sampling problem where
$\tp,\btp$ are transition matrices and $B, \bB$ are observation matrices.
Let $\mu^*(\model)$ and $\mu^*({\bmodel})$ denote, respectively, the optimal strategies for these two different models.
 Let $J_{\mu^*(\model)}(\pi;\model) =  V(\pi;\model)$ and 
$J_{\mu^*(\bmodel)}(\pi;\bmodel) =V(\pi;\bmodel)$ denote the   optimal value functions in (\ref{eq:dp_alg}) corresponding to applying the respective optimal strategies.
Recall  also that the costs in (\ref{eq:cost}) depend on the transition matrix $\tp$. So we will use the notation $C(\pi,u;\model)$ 
and $C(\pi,u;\bmodel)$ to make
the dependence of the cost on the transition matrix explicit.

Introduce the following {\em reverse Blackwell ordering} \cite{Rie91} of observation distributions. Let $B$ and $\bB$ denote
two observation distributions defined as in (\ref{eq:obs}). Then $\bB$ reverse Blackwell dominates $B$ denoted
as 
\beq \bB \bd  B \text{ if }  \bB = B \aB  \label{eq:blackwell}\eeq
where $R= (R_{lm})$ is a stochastic kernel, i.e., $\sum_mR_{lm} = 1$.  This means that $B$ yields more accurate measurements
of the underlying state than $\bB$.

{\em The question we pose is:  How does the optimal cost
$J_{\mu^*(\model)}(\pi;\model)$ vary with transition matrix $\tp$ and observation distribution $B$?} For example, in the quickest  detection optimal sampling
problem, do certain phase-type distributions result in larger total optimal cost compared to other phase-type
distributions?

\begin{framed}
\begin{theorem} \label{thm:tmove}
 Consider two optimal sampling problems with models $\model=(\tp,B)$ and $\bmodel=(\btp,\bB)$, respectively. Assume ${\tp} \succeq \btp$
 with respect to ordering (\ref{eq:mor}),
  $B \bd \bB$ with respect to ordering (\ref{eq:blackwell}), and
  (A1), (A2), (A3)  hold. Then the expected total costs incurred by the optimal sampling
  strategies  satisfy
 $J_{\mu^*(\model)}(\pi;\model) \geq J_{\mu^*(\bmodel)}(\pi;\bmodel)$.
That is, the larger the transition matrix and observation matrix  (with respect to the partial  ordering (\ref{eq:mor}) and (\ref{eq:blackwell})),
the lower the  expected total  cost of the optimal sampling strategy.
 \end{theorem} \end{framed}

The proof is in Appendix \ref{app:tmove}.
Computing the optimal strategies of a POMDP and therefore optimal costs
 is  intractable. Yet, the above
theorem allows us to compare these optimal  costs for different transition and observation matrices.  The implication for quickest  detection with optimal sampling  is that
we can compare the optimal cost for different PH-distributed change times and noise distributions. The implication for sensor scheduling is that we can 
apriori say that certain state dynamics incur a larger overall cost compared to other dynamics and noise distributions.

As a trivial consequence of the theorem, the optimal cost incurred with perfect measurements is always smaller than
that with noisy measurements. Since the optimal sampling problem with perfect measurements is a full observed MDP (or equivalently,
infinite signal to noise ratio), the 
corresponding optimal cost forms a easily computable lower bound to the achievable cost.

\subsubsection*{Examples} Here are examples of transition matrices $\tp ,\btp$ that satisfy (A3) and $\tp \succeq \btp$. \\
{Example 1}. Geometric distributed change time: $\tp = \begin{bmatrix} 1 & 0 \\ 1- \tp_{22} & \tp_{22} \end{bmatrix}$,
$\btp = \begin{bmatrix} 1 & 0 \\ 1- \btp_{22} & \btp_{22} \end{bmatrix}$  where $  \tp_{22} < \btp_{22} $.\\
{Example 2. } PH-distributed change time:  $\tp = \begin{bmatrix} 
1 & 0 & 0 \\ 0.5 & 0.3 & 0.2 \\
0.3 & 0.4 & 0.3 \end{bmatrix}$, $  \btp = \begin{bmatrix}
1 & 0 & 0 \\ 0.9 & 0.1 & 0 \\ 0.8 & 0.2 & 0 \end{bmatrix} $.\\
{Example 3. } Markov chain without absorbing state: $ \tp = \begin{bmatrix} 0.2  & 0.8 \\ 0.1 & 0.9 \end{bmatrix} , \quad \btp =  \begin{bmatrix}  0.8 & 0.2 \\ 0.7 & 0.3 \end{bmatrix}$
 
  Theorem \ref{thm:tmove} applies to all these examples implying that the total cost of the optimal policy with $\tp$ is larger than that with
  $\btp$.

\subsection{Sensitivity to Mis-specified Model}
How sensitive is the total sampling cost to the choice of sampling strategy?
Given two distinct models $\model= (\tp,B)$ and $\bmodel = (\btp,\bB)$ of the optimal sampling problem,
Theorem  \ref{thm:tmove} above compared their optimal costs -- it showed that  $\model \succeq \bmodel \implies J_{\mu^*(\model)}(\pi;\model) \leq J_{\mu^*(\bmodel)}(\pi;\bmodel)$,  where
$\mu^*(\model), \mu^*(\bmodel)$ denote the optimal sampling strategies for models $\model,\bmodel$, respectively.
(where  the ordering $\succeq$ is specified in Sec.\ref{sec:order}). In this section, we establish the following type of sensitivity result 
(where the norm $\|\cdot\|$ is defined in Theorem \ref{thm:sens} below):
\beq
\sup_{\pi\in \I} |J_{\mu^*(\model)}(\pi;\model) - J_{\mu^*(\model)}(\pi;\bmodel)| \leq
K\|\model - \bmodel\|  \label{eq:sens0}\eeq
and,
we will give an explicit representation for the positive constant $K$ in Theorem \ref{thm:sens} below.

 Note the key difference between (\ref{eq:sens0}) 
and  Theorem \ref{thm:tmove}.
In (\ref{eq:sens0}), we are applying the optimal strategy $\mu^*(\model)$ for model   $\model$ 
to the decision problem with a different model  $\bmodel$.
Of course, this results in sub-optimal behavior. But we will show that if the ``distance'' between the two models
$\model,\bmodel$  is small, then the sub-optimality is small --  that is,  the increase in total cost by using
the strategy $\mu^*(\model)$ to the decision problem with model $\bmodel$ is small. 

Define 
\beq \ytt = \inf\{y:   (C_{{\bu}} - C_0)^\p T(e_X,y,u;\model)  \leq 0 \text{ and }
(C_{{\bu}} - C_0)^\p T(e_X,y,u;\bmodel)  \leq 0
\;\forall u,\bu \in \delayseti \} \label{eq:ystar}\eeq
The set depicted in (\ref{eq:ystar}) represents a subset the observation
space $\Y$  for
which the optimal decision is to stop. 
We assume that 
\begin{itemize}
\item[(A7)]  $P(y \leq \ytt) > 0 $. 
\end{itemize}
Assumption  (A7) holds  trivially if the observation distribution $B_{xy}$ (defined in (\ref{eq:obs})) is absolutely continuous with respect to Lebesgue measure on $\reals$,
 i.e., if the density  has support on $\reals$ such as Gaussian noise. (A7) is relevant
 for cases when the observation space is finite or a subset of $\reals$.

\begin{framed}
\begin{theorem} \label{thm:sens} 
Consider two optimal sampling problems  with models $\model= (\tp,B)$ and $\bmodel= (\btp,\bB)$, respectively.
 Let $J_\mus(\pi;\model)$ and
$J_\mus(\pi;\bmodel)$ denote the total costs (\ref{eq:cost}) incurred by these models 
 when using strategy $\mu^*(\model)$. Assume $\model$ and $\bmodel$ satisfy  (A2), (A3), (A4), (A7).
 Then the difference in the total costs is upper-bounded as:
\begin{align} \label{eq:sens}
& \sup_{\pi \in \I}  |J_{\mus}(\pi;\model) - J_{\mus}(\pi;\bmodel)|  \leq \max_i C(e_i,0)   \frac{\|\model-\bmodel\|}{1-\rmm} \\
\text{ where } & 
\rmm  = \max_u \sum_{y \geq \ytt} \sigma(e_X,y,u;\model) 
, \quad  
\|\model-\bmodel\| =  \max_{i,u} \sum_{j,y} \left| B_{jy}\tp^{D_u}\vert_{ij}-\bB_{jy} {\btp}^{D_u}\vert_{ij} \right|. \nonumber
 \end{align}
\end{theorem} \end{framed}
\begin{framed}
 \begin{corollary} \label{cor:kl}
 Consider two optimal sampling problems  with models $\model$ and $\bmodel$, respectively.
 Assume they  have identical transition matrices, but different observation distributions denoted $B,\bB$ (where $B$ is defined in (\ref{eq:obs}).
 Then bound (\ref{eq:sens}) holds with
 \begin{align}
 \|\model-\bmodel\| &=   \sqrt{2 } \max_{i,u} \sum_{j}  \tp^{D_u}\vert_{ij } \,  \left[ D(B_j \| \bB_j)\right]^{1/2} \nonumber \\
 \text{where } 
 D(B_j \| \bB_j) &= \sum_y B_{jy} \ln (B_{jy}/\bB_{jy}) \text{ (Kullback-Leibler Divergence).} \label{eq:kl}
 \end{align}
 In particular, if the observation distributions are   Gaussians with variance $\sigma^2$, $\bsigma^2$, respectively, then
 (\ref{eq:sens}) holds with
$$ \|\model-\bmodel\|  = \left(\frac{\sigma}{\bsigma} - \ln \frac{\sigma}{\bsigma}  - 1\right)^{1/2} $$
\end{corollary} \end{framed}

 The proof of Theorem \ref{thm:sens} and Corollary \ref{cor:kl} are in Appendix \ref{app:sens}.  Corollary \ref{cor:kl}  follows from Theorem \ref{thm:sens} via elementary use of the Pinsker inequality
 that bounds the total variation norm by the Kullback-Leibler Divergence.
Note that  the bound in (\ref{eq:sens}) and (\ref{eq:kl}) is tight in the sense that $\|\model - \bmodel\| = 0$ 
 implies that the performance degradation is zero.  The proof is  complicated by the fact
 that  there is no discount factor\footnote{Instead of (\ref{eq:cost}), if the cost was $ J_\mu(\pizero) = \Ep\left\{\sum_{k=1}^{\kstar-1}  \rho^k C(\pi_{k},u_{k}) 
+ \rho^{\kstar} C(\pi_{{{\kstar}}}, u_{{\kstar}}=0)
 \right\}  $, where the  user defined discount factor $\rho \in [0,1)$, then establishing a bound such as (\ref{eq:sens}) is  straightforward. An artificial discount factor $\rho$ is un-natural in our problem and un-necessary as shown in Theorem \ref{thm:sens} since the problem terminates in finite time with probability one and hence has an implicit discount factor denoted as $\rmm$.}
  in the cost
 (\ref{eq:cost}). However, because the sampling problem terminates with probability one in finite time, it has an implicit
 discount factor -- this is typical in  stochastic shortest path problems that terminate in finite time \cite{Ber00b}. 
 Assumption (A7) implies that $\rmm < 1$.
 The term $1-\rmm$ can be interpreted as a lower bound to the 
 probability of stopping at any given time. Since this is non-zero,
the term $\rmm$ in (\ref{eq:sens}) serves as this implicit discount factor.

 The above result is more than an intellectual curiosity. For optimal sampling problems where
 the  transition matrix or observation distribution do not satisfy 
assumptions  (A5), (A6) or (A7) but are $\epsilon$ close  to satisfying these conditions, the above result  ensures that a monotone strategy
yields near optimal behavior. 
with explicit bound on the performance given by (\ref{eq:sens}) and (\ref{eq:kl}).

\subsection{Stochastic Dominance Properties of  the Bayesian Filter}
 \label{sec:filter}
 This section presents  structural properties of the 
Bayesian filter  (\ref{eq:hmm}) which determines the evolution of the  belief state $\pi$.
Indeed, the proofs of Theorems \ref{thm:main1}-\ref{thm:sens} presented in previous sections
depend on Theorem \ref{thm:filter} given below.
 The results in Theorem \ref{thm:filter} are also of independent interest in Bayesian filtering and prediction.
To compare posterior distributions of Bayesian filters we need to introduce stochastic orders.
We first start with some background definitions.

\subsubsection{Stochastic Orders}
 In order to compare  belief states
we will use the monotone likelihood ratio (MLR)
stochastic order.
 
\begin{definition}[MLR ordering, \cite{MS02}]
Let $\pi_1, \pi_2 \in \I$ be any two belief state vectors.
Then $\pi_1$ is greater than $\pi_2$ with respect to the MLR ordering -- denoted as
$\pi_1 \gr \pi_2$,
 if 
\beq \pi_1(i) \pi_2(j) \leq \pi_2(i) \pi_1(j), \quad i < j, i,j\in \{1,\ldots,X\}. 
\label{eq:mlrorder}\eeq
\end{definition}
Similarly $\pi_1 \lr \pi_2$ if  $\leq$ in (\ref{eq:mlrorder}) is replaced
by a $\geq$. \\
The MLR stochastic order is  useful  since it is closed  under conditional expectations.
That is, $X\gr Y$  implies $\E\{X|\F\} \gr \E\{Y|\F\}$ for any two random variables $X,Y$ and sigma-algebra $\F$
 \cite{Rie91,KR80,Whi82,MS02}. 

\begin{definition}[First order stochastic dominance, \cite{MS02}]
 Let $\pi_1 ,\pi_2 \in \I$.
Then $\pi_1$ first order stochastically dominates $\pi_2$  -- denoted as
$\pi_1 \gs \pi_2$ --
 if 
$\sum_{i=j}^X \pi_1(i) \geq \sum_{i=j}^X \pi_2(i)$  for $ j=1,\ldots,X$.
\end{definition}

The following result is well known \cite{MS02}. It says that MLR dominance implies first order stochastic dominance and gives a necessary and sufficient
condition for stochastic dominance.
\begin{theorem}[\cite{MS02}] \label{res1}
 (i)  Let $\pi_1 ,\pi_2 \in \I$.
Then $\pi_1 \gr \pi_2$ implies $\pi_1\gs \pi_2$.\\
(ii) Let $\mathcal{V}$ denote the set of all $X$ dimensional vectors
$v$ with 
 nondecreasing components, i.e., $v_1 \leq v_2 \leq \cdots
v_X$.
Then $\pi_1 \gs \pi_2$ iff for all $v \in \mathcal{V}$,
 $v^\p \pi_1 \geq v^\p \pi_2$. 
\end{theorem}

For state-space dimension $X =2$, MLR is a complete order and coincides with
first order stochastic dominance.
For state-space dimension $X >2$,
MLR is a  {\em partial order}, i.e., $[\I,\gr]$ is a partially ordered set (poset) since it is not always
possible to order any two belief states $\pi \in \I$.

\subsubsection{Main Result} With the above definitions, we are now ready to state the main result regarding the stochastic dominance
properties of the Bayesian filter. 
\begin{framed}
\begin{theorem}  \label{thm:filter} The following structural properties hold for the Bayesian filtering update $T(\pi,y,u)$ and
normalization measure $\sigma(\pi,y,u)$ defined in (\ref{eq:hmm}):
\begin{enumerate}
\item 
Under (A2), $\pi_1 \gr \pi_2$ implies   $\Tp(\pi_1,y,u)\gr \Tp(\pi_2,y,u)$ holds.
\item Under (A2) and (A3),  $\pi_1 \gr \pi_2$ implies   the normalization measure satisfies
 $\sigp(\pi_1,\cdot,u) \gs \sigp(\pi_2,\cdot,u)$.
 \item Under (A3) and (A5-(i)), the normalization measure $\sigma(\pi,\cdot,u)$ satisfies the following submodular property:
$$\sum_{y\geq \bar{y}}\left[ \sigma(\pi,y,u+1)  - \sigma(\pi,y, u)\right]
 \leq \sum_{y\geq \bar{y}} \left[\sigma(\bp,y,u+1)  - \sigma(\bp,y, u)\right] \text{ for } \pi \gr\bp$$
\item 
For $y,\bar{y} \in \Y$,  $y > \bar{y}$ implies $\Tp(\pi_1,y,u)\gr \Tp(\pi_1,\bar{y},u)$ iff 
(A3) holds.
\item Consider the ordering of transition matrices $\tp \succeq \btp$ defined in (\ref{eq:mor}).
\begin{enumerate}
\item If $\tp \succeq \btp$ then ${\tp}^\p \pi \gr {\btp}^\p \pi$, that is, the one-step Bayesian predictor with transition matrix
$\tp$ MLR dominates that with transition matrix $\btp$.
\item If  $\tp \succeq \btp$ and (A2) holds, then $({\tp^l})^\p \pi \gr ({\btp^l})^\p \pi$ for any positive integer $l$. That is, the $l$-step Bayesian predictor preserves this MLR dominance.
\end{enumerate}
\item Let  $T(\pi,y,u;\tp)$ and $\sigma(\pi,y,u;\tp)$ denote, respectively, the Bayesian filter update and normalization measure
using transition matrix $\tp$.  Then they satisfy the following stochastic dominance property with respect to 
the ordering of $\tp$ defined in (\ref{eq:mor}):
\begin{enumerate}
\item   $\tp \succeq \btp$  implies $T(\pi,y,u;\tp) \gr T(\pi,y,u;\btp)$.
\item
 Under  (A3), $\tp \succeq \btp$ implies $\sigma(\pi,\cdot,u;\tp) \gs \sigma(\pi,\cdot ,u;\btp)$. 
\end{enumerate}
\end{enumerate}
\end{theorem}
\end{framed}
In words, 
Part~1 of the theorem implies that the Bayesian filtering recursion preserves the MLR ordering providing that  the transition matrix is TP2
(A2).  Part~2 says that the normalization measure  preserves first order stochastic dominance providing (A2) and (A3) hold. 
Part~3 shows that the normalization measure is submodular. This is a crucial property in establishing Theorem~\ref{thm:main1}.
Part~4 shows
that under (A3), the larger the observation value, the larger  the posterior distribution (wrt MLR order). 
 Part~5 shows that if starting with two different transition matrices but identical priors, then the optimal predictor with the larger transition matrix
 (in terms of the order introduced in (\ref{eq:mor})) MLR dominates the predictor with the smaller transition matrix.
Part~6 says that same thing about the filtering recursion $T(\pi,y,u)$ and the normalization measure~$\sigma(\pi,y,u)$.

\section{Numerical Examples} \label{sec:numerical}

{\em Example 1. Optimal Sampling  Quickest Detection with Binary Erasure Channel measurements}:
Consider $X=2,Y=3, L=5$, $f=17$, $d=0.4$, $\mc(e_i,1) = 0$, $\mc(e_i,2) = 2.8$,
$$   \{D_1,D_2,D_3,D_4\}  = \{1, 3, 5, 10 \},\;
\tp = \begin{bmatrix} 1 & 0 \\ 0.1 & 0.9  \end{bmatrix}, \; B = \begin{bmatrix}
0.3 & 0.7 & 0 \\ 0 & 0.2 & 0.8 \
\end{bmatrix} $$
The noisy observations of the Markov chain specified by observation probabilities $B$ models a binary non-symmetric erasure channel
\cite{CT06}.
Note that a binary erasure channel is TP2 by construction (all second order minors are non-negative) and so (A3) holds.

 The optimal strategy was computed by forming a grid of 1000 values in the 2-dimensional unit simplex, and then solving the value iteration algorithm 
 (\ref{eq:vi}) over this grid on a horizon $N$ such that $\sup_{\pi} |V_N(\pi) - V_{N-1}(\pi) | < 10^{-6}$.
Figure
 \ref{fig:thm1}(a) shows that when the conditions of Theorem \ref{thm:main1} are satisfied, the strategy is monotone decreasing in
 posterior $\pi(1)$.
To show that the sufficient conditions of Theorem \ref{thm:main1} are useful, Figure
 \ref{fig:thm1}(b) gives an example of when these conditions do not hold, the optimal strategy is no longer monotone.
 Here  $\mc(e_i,1) = 2.8$, $\mc(e_i,2) = 0$ and therefore violates (A1) of Theorem \ref{cor:qd}.

 {\em Example 2. Optimal Sampling Quickest Detection  with Gaussian noise measurements}:
Here we consider identical parameters
to Example 1 except that the observation distribution is Gaussian with $B_{1y} \sim N(1,1)$, $B_{2y} \sim N(2,1)$
and measurement costs are  $\mc(e_i,u) = 1$ for all $i\in \X, u \in \{1,2,3,4\}$.
Since the measurement cost is a constant (A1) and (A4) of Theorem~\ref{cor:qd} hold trivially. As mentioned in Sec.\ref{sec:discussion},
(A3) holds for Gaussian distribution. Therefore Theorem \ref{cor:qd} applies and the optimal strategy $\mu^*(\pi)$ is monotone
decreasing in $\pi(1)$.
Fig.\ref{fig:gp} illustrates the optimal strategy.  
Next, using Theorem \ref{cor:qd2}, the myopic strategies $\bar{\mu}(\pi)$  forms an upper bound
to the optimal strategy $\mu^*(\pi)$ for actions $u \in \{1,\ldots,4\}$.
 We used $\alpha = f/(d(1 - \tp_{21})) $ to satisfy (A7)(i)
for the myopic cost in (\ref{eq:bmu}).
 As a bound
 for the optimal stopping
 region, we used the myopic stopping set $\uStop$ defined in~(\ref{eq:ustop}).  These are plotted in Fig.\ref{fig:gp}(a).

\begin{figure}\centering
\mbox{\subfigure[Gaussian Observation Probabilities]{\epsfig{figure=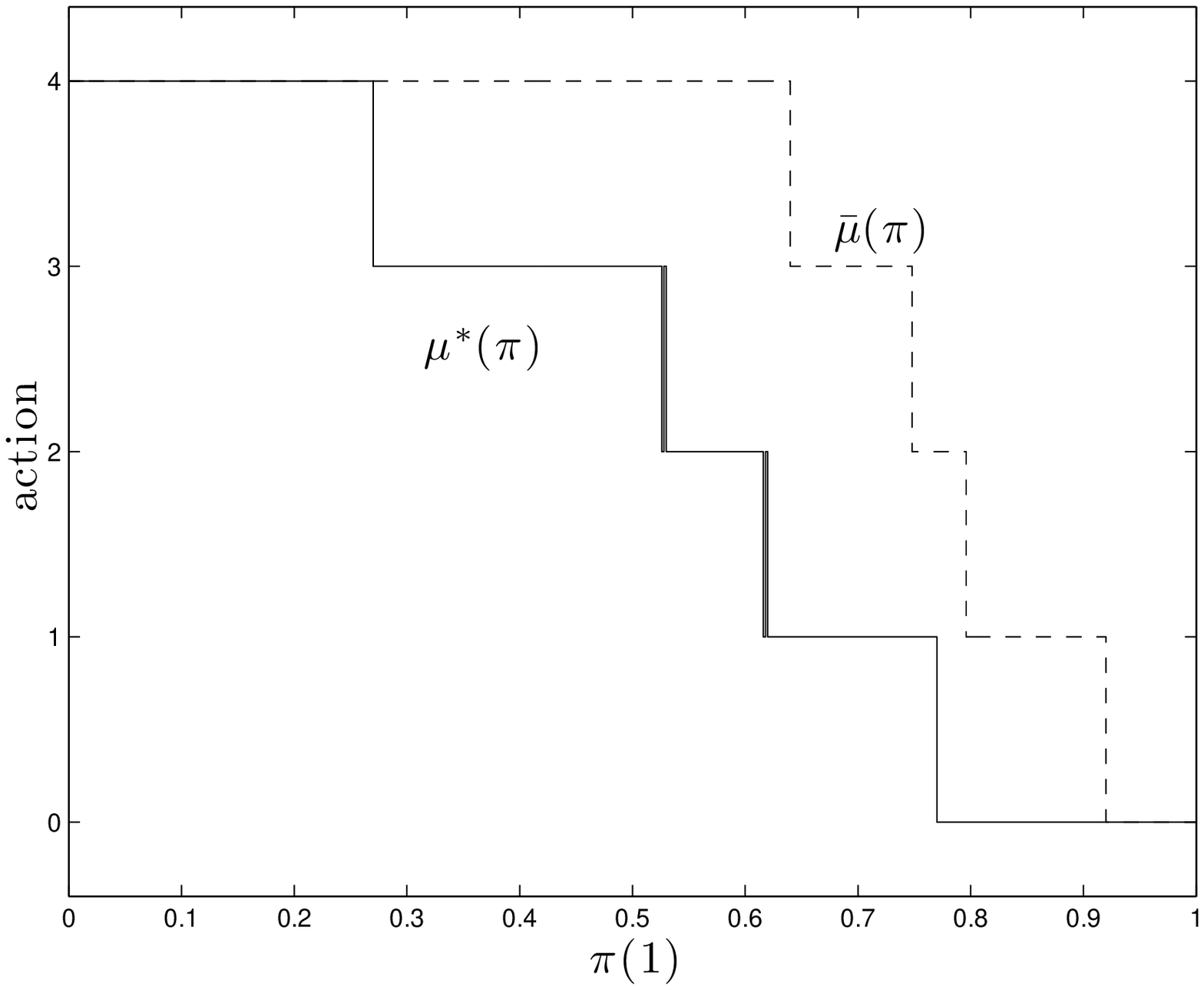,width=0.45\linewidth} } \quad
\subfigure[Poisson Observation Probabilities]{\epsfig{figure=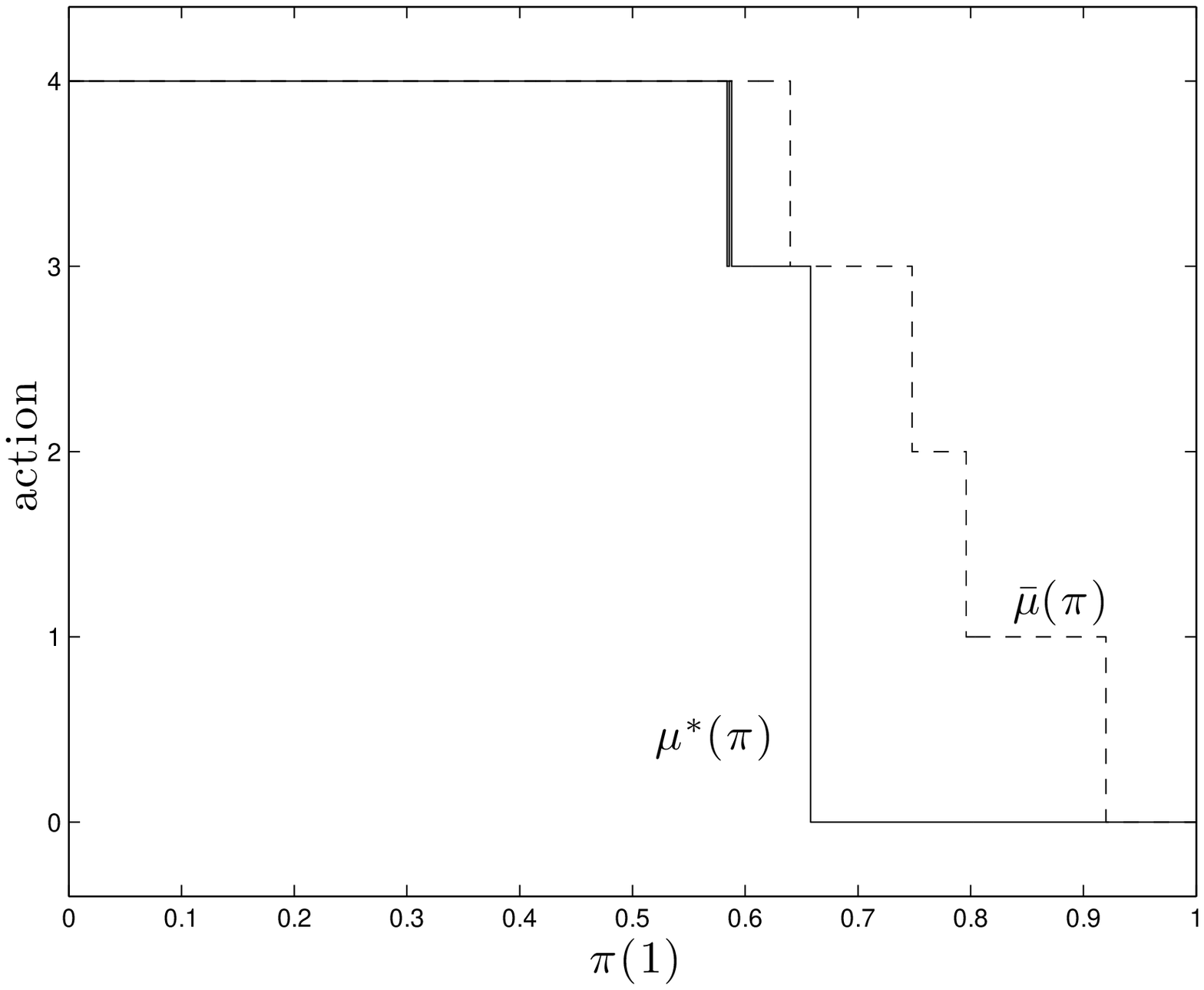,width=0.45\linewidth} }}
\caption{Optimal sampling strategy for action  $u\in \{0 \text{ (announce change)}, 1, 2, 3, 4\}$ for a quickest-change detection
problem with geometric change time. The parameters are specified in Example 2 and 3 in Sec.\ref{sec:numerical}. The optimal strategy
$\mu^*(\pi)$ is monotone decreasing in $\pi(1)$  and is 
 upper  bounded by myopic strategy $\bar{\mu}$  according to Theorem \ref{cor:qd2}.}
\label{fig:gp}
\end{figure}

{\em Example 3. Optimal Sampling Quickest Detection with Markov Modulated Poisson measurements}:
The parameters here are identical to Example 2 except that the observations are generated by a discrete time Markov Modulated Poisson
process. That is, at each time $k$, observations are generated according to the Poisson distribution 
$B_{xy} =  (\lambda_{x})^{y - 1} \frac{e^{-\lambda_x}}{(y-1)!} $ where the rates $\lambda_1 = 1$, $\lambda_2 = 1.5$. Since (A3) holds
for Poisson distribution, Theorem \ref{cor:qd} applies. Fig.\ref{fig:gp}(b) illustrates the optimal strategy. As in Example 2, the myopic strategy $\bar{\mu}(\pi)$ 
forms an upper  bound.

{\em Example 4. Optimal Sampling with Phase-Distributed Change Time}:
Here we consider optimal sampling quickest detection with PH-distributed change time.
Consider a 3-state ($X=3$) Markov chain observed in noise with parameters $f=10$, $d=0.4$, $\mc(e_i,u) = 1$, 
$$ \tp = \begin{bmatrix}  1 & 0 & 0 \\  0.7 & 0.3 & 0 \\  0.3 & 0.4 & 0.3 \end{bmatrix},
B =  \begin{bmatrix} 0.8 & 0.2 & 0 \\  0.1 & 0.8 & 0.1\\  0 & 0.1 & 0.9 \end{bmatrix},  
\{D_1,D_2,D_3,D_4\} = \{1, 2, 4, 5\}.$$
So $\I$ is a 2-dimensional unit simplex.
The optimal strategy was computed by forming a grid of 8000 values in the 2-dimensional unit simplex, and then solving the value iteration algorithm 
 (\ref{eq:vi}) over this grid on a horizon $N$ such that $\sup_{\pi} |V_N(\pi) - V_{N-1}(\pi) | < 10^{-6}$. Fig.\ref{fig:myopic}(a) shows the optimal strategy.
 
 It can be verified that the transition matrix $\tp$ satisfies (A3), (A6) and (A7) for $\alpha = 100$. Also the observation
 distribution $B$ satisfies (A2). 
Therefore Theorem \ref{cor:qd2}  holds and the optimal strategy is upper
 bounded by the myopic strategy $\bmu(\pi)$ defined in (\ref{eq:bmu}). Fig. \ref{fig:myopic}(b) shows the myopic strategy $\bmu(\pi)$. As a bound
 for the optimal stopping
 region, we used the myopic stopping set $\uStop$ defined in~(\ref{eq:ustop}). In Fig.\ref{fig:myopic}(b) these are represented by `0'.

\begin{figure}\centering
\mbox{\subfigure[Optimal Policy $\mu^*(\pi)$]{\epsfig{figure=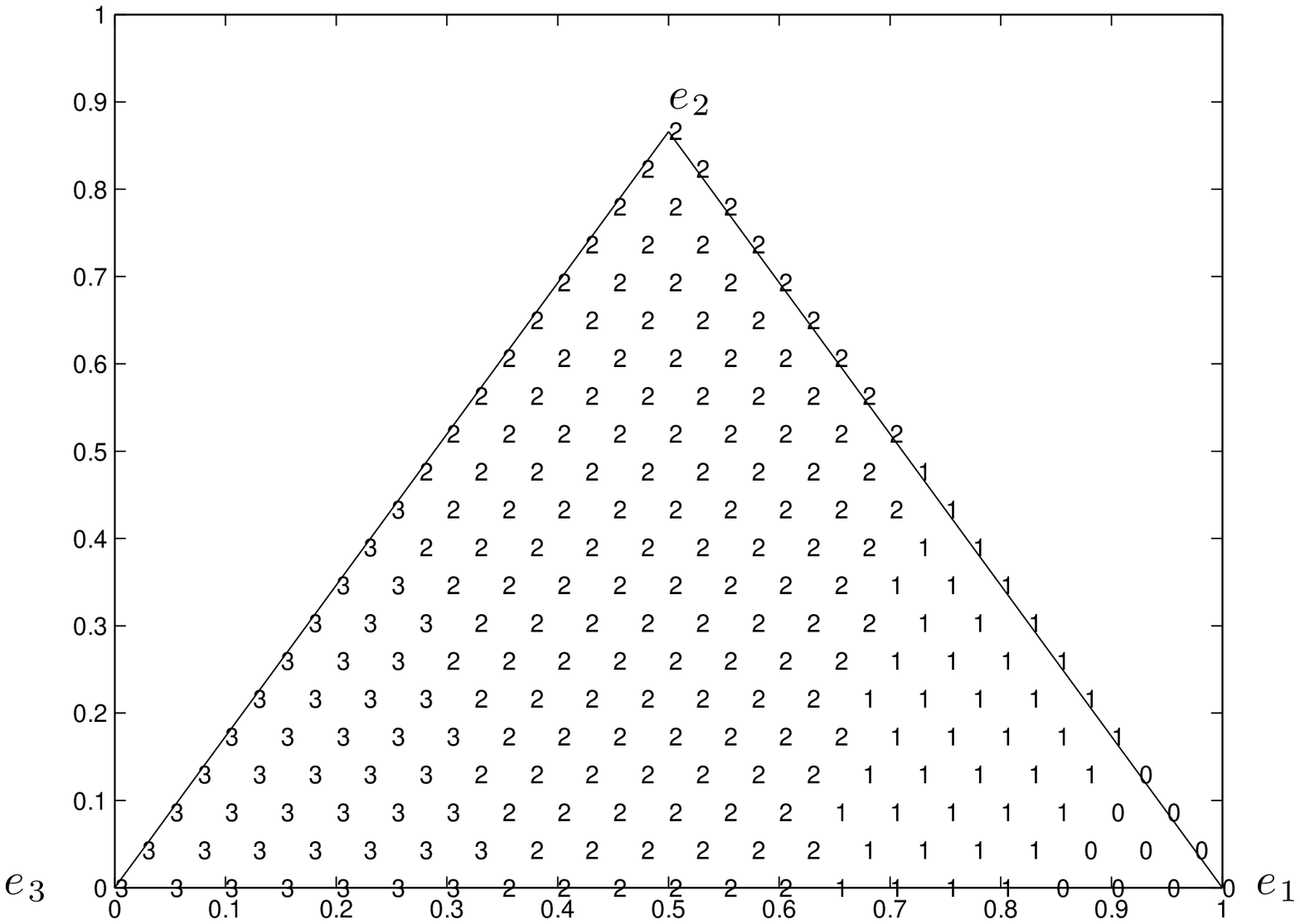,width=0.45\linewidth} } \quad
\subfigure[Myopic Upper Bound  $\bar{\mu}(\pi)$]{\epsfig{figure=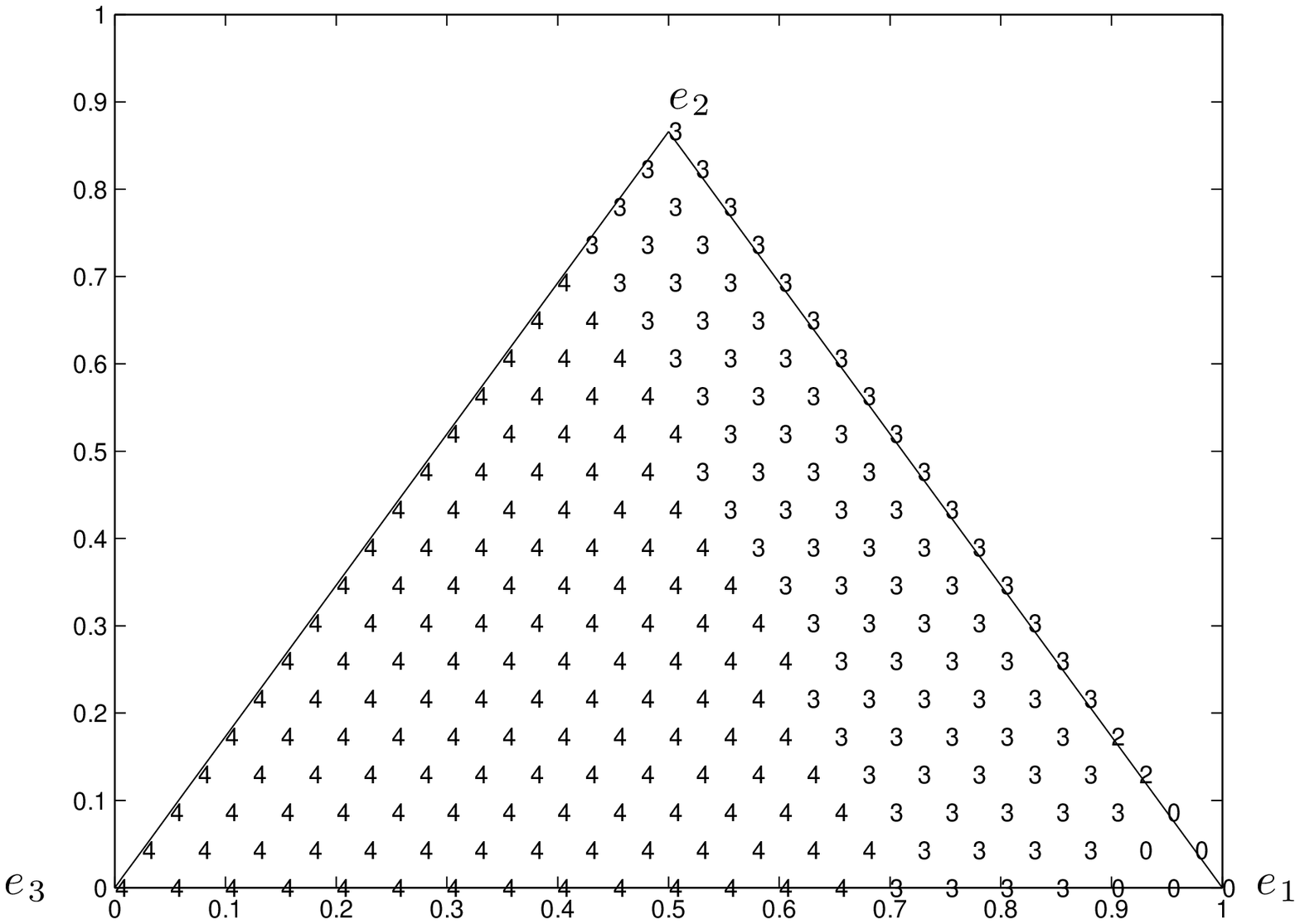,width=0.45\linewidth} }}
\caption{Optimal sampling strategy for action  $u\in \{0 \text{ (announce change)},  1 , 2, 3, 4\}$ for a quickest-change detection
problem with PH-distributed time specified by 3-state Markov chain in Example 4 of Sec.\ref{sec:numerical}. The belief space $\I$ is a two dimensional
unit simplex (equilateral triangle). The optimal strategy
is upper bounded by myopic strategy $\bar{\mu}(\pi)$  according to Theorem \ref{cor:qd2}.}
\label{fig:myopic}
\end{figure}

\section{Discussion} \label{sec:conclusions}
The paper presented  structural results for the optimal sampling strategy of a Markov chain  given noisy measurements.
An example dealing with  quickest change detection with optimal sampling was discussed to motivate the main results.
Such problems are instances of partially observed Markov decision processes (POMDPs) and computing the optimal sampling strategy is 
intractable in general. However, this paper shows that  under reasonable conditions on the sampling costs, transition matrix and 
noise distribution, one can say a lot about the optimal strategy and achievable cost using tools in stochastic dominance and lattice
programming.
There main results were:  Theorems  \ref{thm:main1} and \ref{cor:qd} gave sufficient conditions for the existence of a monotone
optimal sampling strategy (with respect to
the posterior distribution) when the underlying Markov chain had two states. It justified the intuition that one should
make measurements less frequently when the underlying state is away from the target state. Theorem \ref{thm:myopic} and Theorem
\ref{cor:qd2} gave sufficient
conditions for the myopic sampling strategy to form a lower bound or upper bound to the optimal sampling strategy for multi-state Markov chains.
Theorem \ref{thm:tmove} gave a partial ordering for the transition matrix and noise distributions so that the expected cost 
of the optimal sampling strategy decreased as these parameters increased. This yields useful information on the achievable optimal cost
of an otherwise intractable problem.  Theorem \ref{thm:sens}, gave explicit bounds on the sensitivity of the total sampling cost with respect to sampling
strategy in terms of the Kullback Leibler divergence between the noise distributions.  
Theorem \ref{thm:filter} gave several useful structural properties of the optimal Bayesian filtering update including sufficient conditions that preserve monotonicity of the filter with observation, prior distribution, transition matrix and noise
distribution. 

The assumptions (A1-A7) used in this paper are set valued; so even if the precise
parameters (transition probabilities, observation distribution, costs)  are not known, as long as they belong to the appropriate sets, the structural results hold. Thus the results have an inherent robustness.

Finally, it is interesting to note that the results derived in this paper  on sampling control do not apply to 
general  measurement control problems where
 the action affects the observation distribution rather than transition kernel.  The reason is that it is not possible to find two non-trivial stochastic matrices (kernels)
 $B$ and $\bB$ such that the belief updates satisfy  (i) $T(\pi,y;B) \gr T(\pi,y;\bB)$ and normalization
 measure satisfies (ii) $\sigma(\pi,\cdot,B) \gs \sigma(\pi,\cdot;\bB)$. In \cite{Lov87}, it is claimed that if $B$ TP2 dominates $\bB$ then
 (i) and (ii) hold. However, we have found that  the only examples of stochastic kernels that satisfy the TP2 dominance are the trivial example
  $B = \bB$.
 In our paper, which deals with sampling control, the ordering (\ref{eq:mor}) was  constructed so  that two transition matrices $\tp$ and $\btp$ satisfy (i) and (ii) with $B,\bB$
 replaced by $\tp, \btp$. This ordering
 was used in Assumption (A6).

\appendix

\subsection{Value Iteration Algorithm} \label{sec:vi}

The proof of  the structural results in this paper will use the value iteration algorithm \cite{HL96}.
Let $n=1,2,\ldots,$ denote iteration number. The value iteration algorithm
proceeds as follows:
\begin{align}
 V_{n+1}(\pi) &= \min_{u \in \U} Q_{n}(\pi,u),
 \text{ where }   Q_n(\pi,u) =  C(\pi,u)
+ \sum_{y \in \Y}  V_n\left( \Tp(\pi ,y,u) \right) \sigp(\pi,y,u),\nonumber\\
\text{ and } &  Q_n(\pi,0) =  C(\pi,0) \text{ initialized by } V_0(\pi) = 0. \label{eq:vi}
\end{align}

Let $\mathcal{B}(X)$ denote the set of bounded real-valued functions on $\I$.
Then for any $V$ and  $\tilde{V} \in \mathcal{B}(X)$, define the sup-norm metric
$\sup\|V(\pi) - \tilde{V}(\pi)\|$, $\pi \in \I$. Then $\mathcal{B}(X)$ is a Banach
space. The value iteration algorithm (\ref{eq:vi}) will generate a sequence of value functions
$\{V_k\} \subset \mathcal{B}(X)$ that will converge uniformly (sup-norm metric) as $k\rightarrow \infty$ to $V(\pi) \in \mathcal{B}(X)$, the optimal value
 function of Bellman's equation.
However, since the belief state space $\I$ is an uncountable  set, the
value iteration algorithm (\ref{eq:vi}) do not 
translate into practical solution methodologies as
$V_k(\pi) $ needs to be evaluated at each $\pi \in \I$, an
uncountable set. Nevertheless, the value iteration algorithm provides a natural method
for proving our results on the structure of the optimal strategy via mathematical induction.

\subsection{Proof of Theorem \ref{thm:main1}} \label{app:main1}

To prove the existence of a monotone optimal strategy,
we will show that
 $Q(\pi,u)$  in   (\ref{eq:dp_alg}) is a submodular function on  the poset $[\I,\gr]$. Note that  $[\I,\gr]$ is a lattice since
given any two belief states $\pi_1,\pi_2 \in \I$, $\sup\{\pi : \pi \lr \pi_1, \pi \lr \pi_2\}$ and $\inf\{\pi : \pi \gr \pi_1, \pi \gr \pi_2\}$ lie in $\I$.
For $X=2$, $\I$ is  the unit interval [0,1]  and in this case $[\I,\gr]$ is a chain (totally ordered set).

\begin{definition}[Submodular function \cite{Top98}] \label{def:supermod} 
 $f:\I \times \u \rightarrow \reals$  is  submodular (antitone differences)
if 
$f(\pi,u) - f(\pi,\bar{u}) \leq f(\tilde{\pi},u)-f(\tilde{\pi},\bar{u})$, for
$\bar{u} \leq u$, $\pi \gr \tpi$.
\end{definition}


The following result says that for a submodular function  $Q(\pi,u)$, 
$\mu^*(\pi)=\argmin_u Q(\pi,u)$  is increasing in its argument $\pi$. This will be used to prove the existence of a monotone optimal strategy
in Theorem \ref{thm:main1}.

\begin{theorem}[\cite{Top98}] \label{res:monotone}
\label{res:supermod}  If $f:\I\times \U \rightarrow \reals$ is submodular, then
there exists a $\mu^*(\pi) = \argmin_{u\in \U} f(\pi,u)$, that  is MLR increasing on $\I$,
i.e., $\tpi \gr {\pi} \implies \mu^*(\pi) \leq \mu^*(\tpi)$. \qed
\end{theorem}

Finally, we state the following result.
\begin{theorem} The sequence of value function $\{V_n(\pi)$, $ n=1,2,\ldots\}$,  generated by the value iteration algorithm (\ref{eq:vi}), and 
optimal value function $V(\pi)$ defined in (\ref{eq:dp_alg}) satisfy:\\
(i) $V_n(\pi)$  and $V(\pi)$ are concave in $\pi \in I$. \\
(ii)  Under (A1), (A2), (A3), $V_n(\pi)$ and $V(\pi)$ are increasing in $\pi$ with respect to the MLR stochastic order on $\I$. \qed
\label{lem:concave}
\end{theorem}
Statement (i) is well known for POMDPs, see \cite{Cas98} for a tutorial description. Statement (ii) is proved in \cite[Proposition 1]{Lov87} using
mathematical induction on the value iteration algorithm.

{\bf Proof}: With the above preparation, we present the proof of Theorem \ref{thm:main1}.

The first claim follows from the general result that the stopping set $\Stop$  for a POMDP is always a convex subset of $\I$ -- see  Theorem \ref{thm:myopic}. Of course, a one dimensional convex set is an interval and since $e_1 \in \Stop$, it follows that the interval $\Stop = (\pi_1^*,1]$.

In light of the first claim, the optimal strategy is of the form
$$ \mu^*(\pi) =  \begin{cases} 0  & \pi \in \Stop \\
					\argmin_{u \in \delayseti}  Q(\pi,u)  &   \pi \in \I - \Stop \end{cases} $$

So to  prove the second claim,  we only need to focus  on belief states in the interval $\I - \Stop = [0,\pi_1^*]$
and consider actions $u\in \delayseti$.  
To prove that  $\mu^*(\pi)$ is MLR increasing in $\pi \in \I -\Stop$, from Theorem \ref{res:monotone} we  need to prove that $ Q(\pi,u)$ is submodular, that is
$$ Q(\pi,u) - Q(\pi,\bu) - Q(\bpi,u)+ Q(\bpi,\bu) \leq 0, \quad u > \bu, \; \pi \gr \bpi .$$
From (\ref{eq:dp_alg}), the left hand side of the above expression is
\begin{align}
&C(\pi,u) - C(\pi,\bu) - C(\bp,\bu) + C(\bp,\bu)   \nonumber\\
+ & \sum_y V(T(\pi,y,u)) \left[\sigma(\pi,y,u) - \sigma(\pi,y,\bu) - \sigma(\bpi,y,u) + \sigma(\bp,y,\bu) \right]  \nonumber \\
  + & \sum_y \left[V(T(\pi,y,u)) - V(T(\bp,y,u))\right] \sigma(\pi,y,u) 
+ \sum_y  \left[V(T(\pi,y,u)) - V(T(\pi,y,\bu))\right] \sigma(\pi,y,\bu) \nonumber \\ 
 +  &  \sum_y\left[V(T(\bpi,y,\bu)) - V(T(\pi,y,u))\right] \sigma(\bpi,y,\bu) \label{eq:rhs}
\end{align}
Since the cost is submodular by (A4),  the first line of (\ref{eq:rhs}) is negative.
Since $V(\pi)$ is MLR increasing from Theorem \ref{lem:concave} and $T(\pi,y,u)$ is MLR increasing in $y$ from Theorem \ref{thm:filter}(4),
it follows that $V(T(\pi,y,u))$ is MLR increasing in $y$. Therefore,
since $\sigma(\pi,\cdot,u)$ is submodular from Theorem \ref{thm:filter}(3), the second line of (\ref{eq:rhs}) is negative.

It only remains to prove that the third and fourth lines of (\ref{eq:rhs}) are negative.
From statements (1) and (6) of Theorem \ref{thm:filter}, it follows that $T(\pi,y,u) \gr T(\pi,y,\bu) \gr T(\bp,y,\bu)$ and 
$T(\pi,y,u)  \gr T(\bp,y,u) \gr T(\bp,y,\bu)$.  
Now we use the assumption that $X=2$. So the belief  state space $\I$ is a one dimensional simplex
that can be represented by $\pi(2) \in [0,1]$. So below we represent  $\pi$, $T(\pi,y,u)$, etc.\ by their  second elements.
Therefore using concavity of $V(\cdot)$, we can express the last two summations in (\ref{eq:rhs}) as follows:
\begin{align*}
V(T(\pi,y,u)) - V(T(\pi,y,\bu)) &\leq  \left[T(\pi,y,u) - T(\pi,y,\bu)\right]  \frac{V(T(\pi,y,u)) - V(T(\bp,y,u))}{T(\pi,y,u) - T(\bp,y,u)} \\
V(T(\bp,y,\bu)) - V(T(\pi,y,u)) &\leq \frac{T(\bp,y,\bu) -  T(\bp,y,u)}{T(\pi,y,\bu) - T(\bp,y,u)} \left[ V(T(\pi,y,u)) - V(T(\bp,y,u)) \right]  \\ & + V(T(\bp,y,u)) - V(T(\pi,y,u))
\end{align*}
Using these expressions, the summation  of the last two lines of  (\ref{eq:rhs}) are upper bounded by
\begin{multline} \sum_y
\left[V(T(\pi,y,u)) - V(T(\bp,y,u))\right] \biggl[ \sigma(\bpi,y,u) + \frac{T(\pi,y,u) - T(\pi,y,\bu)}{T(\pi,y,u) - T(\bp,y,u)} \sigma(\pi,y,\bu) \\
+ \frac{T(\bp,y,\bu) - T(\pi,y,u)}{T(\pi,y,u) - T(\bp,y,u)} \sigma(\bp,y,\bu) \biggr] \label{eq:subex} \end{multline}
Since $V(\pi)$ is MLR increasing (Theorem \ref{lem:concave}) and $T(\pi,y,u) \gr T(\bpi,y,u)$ (using the fact that $\pi \gr \bpi$ and Statement 1 of Theorem \ref{thm:filter}),
clearly $V(T(\pi,y,u)) - V(T(\bp,y,u)) \geq 0$.
The term in square brackets in (\ref{eq:subex}) can be expressed as (see \cite{Alb79})
$$ \frac{B_{2y} B_{1y} (\pi - \bp)  (\tp^{D_2}|_{22} - \tp^{D_2}|_{12} - \tp^{D_1}|_{22} - \tp^{D_1}|_{12})}{\sigma(\pi,y,u) [T(\pi,y,u)-T(\bp,y,u)]} $$
By Assumption (A5)(ii)  the above term is negative. Hence (\ref{eq:rhs}) is negative, thereby concluding the proof.

\subsection{Proof of Theorem \ref{lem:transform}} \label{app:transform}

Statement 1: Consider Bellman's equation (\ref{eq:dp_alg}) and  define $\uV(\pi) = V(\pi) - \alpha C(\pi,L)$.
It  is easily checked that $\uV(\pi)$ satisfies Bellman's equation with costs $C(\pi,u)$ replaced by $\uC(\pi,u)$ defined in (\ref{eq:transform}).
Also  since the term being subtracted, namely, $\alpha C(\pi,L)$  is functionally independent of the minimization variable $u$,  the argument of the minimum of (\ref{eq:dp_alg}), which is 
the optimal strategy $\mu^*(\pi)$,  is unchanged.

Statement 2: Since our aim is to transform the delay cost to yield a MLR increasing submodular transformed cost, for notational convenience assume the measurement cost $\mc(x,u) = 0$.
From its definition in (\ref{eq:transform}), straightforward computations yield
that the transformed cost is
\begin{align*}
\uC(\pi,0) &= \f ^\p \pi -  \alpha d e_1^\p (I+\tp + \cdots \tp^{D_L-1}) \pi \\
 \uC(\pi,u) &= d e_1^\p \left( (1-\alpha) I + \alpha \tp^{D_L} \right)^\p  (I + \tp + \cdots + \tp^{D_u-1})^\p \pi, \quad u \in \delayseti 
 \end{align*}
 So clearly for $\alpha \geq 0$,   $\uC(e_1,0) \leq \uC(e_2,0)$, and so $\uC(\pi,0)$ is  MLR increasing.
 
We now give conditions for   $\uC(\pi,u)$, for $u \in \delayseti$ to be MLR increasing in $\pi \in \I$.
By (A2), $(I + \tp + \cdots + \tp^{D_u-1})^\p \pi$ is MLR increasing in $\pi \in \I$. So for $\uC(\pi,u)$ to be MLR increasing in $\pi$,
 it suffices to choose $\alpha$ so that
the elements of  $\left( (1-\alpha) I + \alpha \tp^{D_L} \right) e_1$ are increasing.
Given the structure of $\tp$ in (\ref{eq:tpg}), it follows that 
$$\left( (1-\alpha) I + \alpha \tp^{D_L} \right) e_1 = \begin{bmatrix} 1 \\  \alpha(1- \tp_{22}^{D_L})  \end{bmatrix} $$
So choosing $\alpha \geq 1/(1-\tp_{22}^{D_L})$ is sufficient for   $\uC(\pi,u)$ to be MLR increasing in $\pi$ for $u \in \delayseti$ and
therefore for $u \in \{0,1,\ldots,L\}$.

Next for the transformed cost $\uC(\pi,u)$ to be submodular for $u\in \delayseti$, we  require $\uC(\pi,u+1)  - \uC(\pi,u)$ to be MLR decreasing in $\pi$.
Straightforward computations yield for $u \in \delayseti$,
$$ \uC(\pi,u+1)  - \uC(\pi,u) = d  e_1^\p \left( (1-\alpha) I + \alpha \tp^{D_L} \right)^\p  
(\tp^{D_u}+   \tp^{D_u+1} + \cdots + \tp^{D_{u+1}})^\p \pi $$
So for $\uC(\pi,u)$ to be submodular,
 it suffices to choose $\alpha$ so that
the elements of  $\left( (1-\alpha) I + \alpha \tp^{D_L} \right) e_1$ are decreasing, i.e.,  $\alpha \leq 1/(1-\tp_{22}^{D_L})$.

Therefore choosing   $\alpha =1/(1-\tp_{22}^{D_L}) $ is sufficient for the transformed cost $\uC(\pi,u)$ to be both MLR increasing for $u\in \{0,1,\ldots,L\}$ and submodular
for $u\in \delayseti$ on the poset $[\I-\Stop,\gr]$.  

\subsection{Proof of Theorem \ref{thm:myopic}} \label{app:myopic}

{\em Statement 1}:
The proof of convexity of the stopping set $\Stop$ follows from arguments in \cite{Lov87a}. We repeat this for completeness here.
Pick any two belief states $\pi_1,\pi_2 \in \Stop$. To demonstrate convexity of $\Stop$,
we need to show for any $\lambda \in [0,1]$,  $\lambda \pi_1 + (1-\lambda) \pi_2 \in \Stop$.
Since $V(\pi)$ is concave (by Theorem~\ref{lem:concave} above), it follows from (\ref{eq:dp_alg}) that
\begin{align}
V(\lambda \pi_1 + (1-\lambda) \pi_2) &\geq \lambda V(\pi_1) + (1-\lambda) V(\pi_2) \nonumber\\
&= \lambda Q(\pi_1,0) + (1-\lambda) Q(\pi_2,0)  \text{ (since $\pi_1,\pi_2 \in \Stop$) } \nonumber\\
&= Q(\lambda \pi_1 + (1-\lambda) \pi_2,0 ) \text{ (since $Q_{1}(\pi,0)$ is linear in $\pi$) }\nonumber \\
& \geq V(\lambda \pi_1 + (1-\lambda) \pi_2) \text{ (since $V(\pi)$ is the optimal value function) } \label{eq:convexregion}
\end{align}
Thus all the inequalities above are equalities, and $\lambda \pi_1 + (1-\lambda) \pi_2 \in 
\Stop$.

{\em Statement 2}: Since the costs $C(\pi,u)$ are non-negative, so is $V(\pi)$ in (\ref{eq:dp_alg}).  So from (\ref{eq:dp_alg}).,
$C(\pi,0) \leq C(\pi,u) \implies
Q(\pi,0) \leq Q(\pi,u) \implies \pi \in \Stop$. Therefore $\uStop \subset \Stop$.

{\em Statement 3}:  The proof is similar to \cite[Proposition 2]{Lov87} with the important difference that 
in \cite{Lov87} the  TP2 ordering of transition matrices is used instead of  (A6).
However, the TP2 ordering (see \cite{Lov87} for definition) does not yield any non-trivial example. 

Since $D_u < D_{u+1}$,  (A6) implies $\tp^{D_u} \succeq  \tp^{D_{u+1}}$.
So by Statement 6(a) of Theorem \ref{thm:filter}, for $u > \bu$,
$T(\pi,y,u) \lr T(\pi,y,\bu)$. By Theorem~\ref{lem:concave},  $V(\pi)$ is MLR increasing in $\pi$.
Therefore $V(T(\pi,y,u)) \leq  V(T(\pi,y,\bu))$.
  So
$$\sum_y V(T(\pi,y,u) \sigma(\pi,y,u) \leq  \sum_y V(T(\pi,y,\bu) \sigma(\pi,y,u) , \quad  u > \bu. $$
Since $T(\pi,y,\bu)$ is  MLR increasing in $y$   (Statement 4 of Theorem \ref{thm:filter})
and $V(\pi)$ is MLR  increasing in $\pi$, clearly
 $V(T(\pi,y,\bu))$ is increasing in $y$.
 Also (A6) implies $\tp^{D_u} \succeq  \tp^{D_{u+1}}$ and so
 $ \sigma(\pi,\cdot,u)  \ls  \sigma(\pi,\cdot,\bu)$ from Statement 6(b) of Theorem \ref{thm:filter}. So
$$\sum_y V(T(\pi,y,\bu)\, \sigma(\pi,y,u) \leq \sum_y V(T(\pi,y,\bu) \,\sigma(\pi,y,\bu), \quad  u > \bu.$$
Therefore, $\sum_y V(T(\pi,y,u) \sigma(\pi,y,u) \leq \sum_y V(T(\pi,y,\bu) \sigma(\pi,y,\bu)$ which is equivalent to
$Q(\pi,u) - Q(\pi,\bu)  \leq C(\pi,u) - C(\pi,\bu)$.
Then \cite[Lemma 2.2]{Lov87} implies that the minimizers of $Q(\pi,u) - Q(\pi,\bu)$ are larger than that of $C(\pi,u) - C(\pi,\bu)$.
That is $\mu^*(\pi) \geq \bar{\mu}(\pi)$ for $\pi \in \I$.

{\em Statement 4}:  By (A4),  $C(\pi,u)$ is submodular on the poset $[\I,\gr]$. So using Theorem \ref{res:monotone}
it follows that $\umu(\pi)$ is MLR increasing.

\subsection{Proof of Theorem \ref{cor:qd2}} \label{app:qd2}
Statements 1 and 2 follows directly from Theorem \ref{thm:myopic}.


Statement 3:  We prove this in the following steps.

{\em Step 1}. $\mu^*(\pi)$ remains invariant with transformed cost:
For costs $\Cb(\pi,u)$  Bellman's equation yields the same optimal strategy $\mu^*(\pi)$ as costs $C(\pi,u)$. To see this,
consider Bellman's equation (\ref{eq:dp_alg}) and  define $\Vb(\pi) = V(\pi) + \alpha d e_1^\p \tp^\p \pi$.
It  is easily checked that $\Vb(\pi)$ satisfies Bellman's equation with costs $C(\pi,u)$ replaced by $\Cb(\pi,u)$ defined in (\ref{eq:transform}).
Also  since the term being added, namely $\alpha d e_1^\p \tp^\p \pi$ is functionally independent of the minimization variable $u$,  the argument of the minimum of (\ref{eq:dp_alg}), which is 
the optimal strategy $\mu^*(\pi)$,  is unchanged.

{\em Step 2}. $\Cb(\pi,u)$ is MLR decreasing:
We show that (A7) implies that $\Cb(\pi,u)$ is MLR decreasing, i.e., $-\Cb(\pi,u)$ is MLR increasing and satisfies (A1).

First consider $\Cb(\pi,0)$. Note
$\Cb(e_1,0) = \alpha d$, and $\Cb(e_i,0) = f + \alpha \tp_{i1}$ for $i > 1$.
So $\Cb(e_1,0) \geq \Cb(e_2,0)$ if $\alpha > f/(d(1-\tp_{21}))$.
Since $
\alpha > 0$ and $\tp$ is TP2 (Assumption A3), $\tp_{i1} > \tp_{i+1,1}$. So clearly $\Cb(e_i,0) \geq  \Cb(e_{i+1},0)$ for $i \geq 2$.

Next consider $\Cb(\pi,u)$, $u\in \delayseti$. Note $\Cb(e_1,u) = d D_u$ and 
\beq \Cb(e_i,u) = 
d (\tp_{i1} + \tp^2|_{i1} + \cdots + \tp^{D_u-1}|_{i1}) + \alpha d (\tp_{21} - \tp^{D_u+1}|_{21} ). \label{eq:cbi} \eeq
Clearly
$\Cb(e_i,u) < d(D_u-1) + \alpha d(\tp_{i1} - \tp^{D_u+1}|_{i1}) $.
Also $\tp_{i1} -  \tp^{D_u+1}|_{i1} \leq 0$ since $\tp^{D_u+1}|_{i1} = \tp_{i1} + \sum_{j>1} \tp_{ij} A^{D_u}|_{j1}$. Therefore for 
non-negative $\alpha$, $\Cb(e_1,u) \geq \Cb(e_i,u)$, $i\geq 2$.
Also for $\Cb(e_i,u ) \geq \Cb(e_{i+1},u)$, $i\geq 2$,  clearly from (\ref{eq:cbi}) it follows that (A7)(ii) is sufficient.

{\em Step 3}: $\Cb(\pi,u)$, $u\geq 1$,  is submodular (satisfies (A4)). This follows similar to Step 2.

{\em Step 4}: With the above three steps, we can now apply Theorem \ref{thm:myopic}, except that $\Cb(\pi,u)$ is MLR decreasing
instead of MLR increasing as required by (A1).  By a very similar proof to Theorem \ref{thm:myopic}, it follows that $\bmu(\pi) \geq \mu^*(\pi)$.

Statement 4: Follows trivially from Statement 3 for the $X=2$ case.

\subsection{Proof of Theorem \ref{thm:tmove}} \label{app:tmove}

{\bf Part 1}: We first prove that dominance of transition matrices $\tp \succeq \btp$ (with respect to (\ref{eq:mor})) results in dominance of optimal costs, i.e., $V(\pi;\tp) \geq V(\pi;\btp)$.
The proof  is by induction. $V_0(\pi;\tp) \geq V_0(\pi;\btp) = 0$ by the initialization of the value iteration algorithm (\ref{eq:vi}). 
Next, to prove the inductive step
assume  that $V_n(\pi; \tp) \geq V_n(\pi;\btp)$ for $\pi \in \I$.
By Theorem \ref{lem:concave}(ii), under (A1), (A2), (A3), $V_n(\pi;\tp)$ and $V_n(\pi;\btp)$ are MLR increasing in $\pi \in \I$.
From Statement 6(a) of Theorem \ref{thm:filter},  it follows that  $T(\pi,y,u;\tp) \gr T(\pi,y,u;\btp)$. This  implies
$$V_n(T(\pi,y,u;\tp) ; \tp)    \geq V_n(T(\pi,y,u;\btp) ; \tp)  , \quad \tp \succeq \btp.$$
Since $V_n(\pi; \tp) \geq V_n(\pi;\btp) \; \forall \pi \in \I$ by assumption, clearly  
$V_n(T(\pi,y,u,\btp); \tp) \geq V_n(T(\pi,y,u,\btp);\btp)$.
Therefore $$V_n(T(\pi,y,u;\tp) ; \tp)    \geq V_n(T(\pi,y,u;\btp) ; \tp)  \geq V_n(T(\pi,y,u,\btp);\btp),  \quad \tp \succeq \btp.  $$

Under (A2), (A3),  Statement 4 of Theorem \ref{thm:filter} says that  $T(\pi,y,u;\tp) ; \tp) $ is MLR increasing in $y$. Therefore, $V_n(T(\pi,y,u;\tp) ; \tp)$ is increasing in $y$.
Also from Statement 2 of Theorem \ref{thm:filter},  $\sigma(\pi,\cdot,u;\tp) \gs \sigma((\pi,\cdot ,u;\btp)$  for $\tp \succeq \btp$.
Therefore, 
\beq \sum_y V_n(T(\pi,y,u;\tp) ; \tp) \sigma(\pi,\cdot,u;\tp) \geq \sum_y V_n(T(\pi,y,u;\btp) ; \btp) \sigma(\pi,\cdot,u;\btp)  .
\label{eq:s1}
\eeq

Next, we claim that under (A1) and (A2), $\tp \succeq \btp$ implies that $C(\pi,u;\tp) \geq C(\pi,u,;\btp)$. This follows since $c(e_i,u)$
defined in (\ref{eq:cost})  has increasing components by
(A1) and  $({\tp^l})^\p \pi \gr ({\btp}^l)^\p \pi $ (Statement 5(b), Theorem \ref{thm:filter}). Therefore, 
$c_u ^\p ({\tp^l})^\p \pi\geq  c_u ^\p({\btp}^l)^\p \pi $ implying that $C(\pi,u;\tp) \geq C(\pi,u,\btp)$.
This together with (\ref{eq:s1})  implies
$$C(\pi,u;\tp) + \sum_y V_n(T(\pi,y,u;\tp) ; \tp) \sigma(\pi,\cdot,u;\tp) \geq C(\pi,u,\btp) + \sum_y V_n(T(\pi,y,u;\btp) ; \btp) \sigma(\pi,\cdot,u;\btp) $$
Minimizing both sides with respect to action $u$ yields $V_{n+1}(\pi; \tp) \geq V_{n+1}(\pi;\btp)$ and concludes the induction argument.

{\bf Part 2}: Next we show that dominance of observation distributions $B \bd \bB$ (with respect to the order (\ref{eq:blackwell})) results  in dominance of the optimal costs, namely 
$V(\pi;B) \geq V(\pi,\bB)$.
 Let $T(\pi,y,u)$ and $\bT(\pi,y,u)$ denote  the Bayesian filter update with observation $B$ and $\bB$, respectively, and let
$\sigma(\pi,y,u)$ and $\bsigma(\pi,y,u)$ denote the corresponding normalization measures.

Then for $a\in \Y$,
$$ T(\pi,a,u ) =   \sum_{y \in \Y} \bT(\pi,y,u) \frac{\bsigma(\pi,y,u)}{\sigma(\pi,a,u)} P(a|y) 
\quad \text{ and } \sigma(\pi,a,u) = \sum_{y \in \Y} \bsigma(\pi,y,u) P(a|y).$$
Therefore, $\frac{\bsigma(\pi,y,u)}{\sigma(\pi,y,a)} P(a|y) $ is a probability measure wrt $y$.
Since from Theorem \ref{lem:concave},  $V_n(\cdot)$ is concave for $\pi \in \I$, using Jensen's inequality it follows that
\begin{align}
V_n(T(\pi,a,u) ;\bB) & = V_n \left(\sum_{y \in \Y} \bT(\pi,y,u) \frac{\bsigma(\pi,y,u)}{\sigma(\pi,a,u)} P(a|y); \bB \right)
\geq \sum_{y \in \Y}  V_n (\bT(\pi,y,u);\bB) \frac{\bsigma(\pi,y,u)}{\sigma(\pi,a,u)} P(a|y) \nonumber \\
\text{ implying }&  \sum_{a}  V_n(T(\pi,a,u) ;\bB) \sigma(\pi,a,u) \geq
\sum_{y} V_n(\bT(\pi,y,u);\bB) \bsigma(\pi,y,u). \label{eq:blackwell2}
\end{align}

With the above inequality, the proof of the theorem follows by mathematical induction using the value iteration algorithm~(\ref{eq:vi}).
Assume $V_n(\pi;B) \geq V_n(\pi;\bB)$ for $\pi \in  \I$.
Then 
\begin{align*}C(\pi,u) +  \sum_a V_n(T(\pi,a,u);B) \sigma(\pi,a,u) & \geq
C(\pi,u) +  \sum_a V_n(T(\pi,a,u);\bB) \sigma(\pi,a,u) \\  & \geq C(\pi,u) + \sum_y V_n(\bT(\pi,y,u);\bB) \bsigma(\pi,y,u)  
 \end{align*}
where the second inequality follows from  (\ref{eq:blackwell2}).
Thus  $V_{n+1}(\pi;B) \geq V_{n+1}(\pi;\bB)$. This completes the induction step. Since value iteration algorithm (\ref{eq:vi}) converges uniformly,
$V(\pi;B) \geq V(\pi;\bB)$ thus proving the theorem.

\subsection{Proof of Theorem \ref{thm:sens}} \label{app:sens}
Define the set of belief states
$\Stopb = \cap_u \{\pi : (C_u-C_0)^\p \pi \geq 0\}$. Clearly $\Stopb \subseteq \Stop$. Let us characterize the set of observations
such that the Bayesian filter update $T(\pi,y,u;\model)$ lies in $\Stopb$ for any action $u$. Accordingly, define
\beq
 \Rp  = \{ y: (C_{\bu} - C_0)^\p T(\pi,y,u;\model) \geq 0, \; \forall u,\bu \in \delayseti\}, \quad
 \yp = \inf \{y: y \in \Rpc\} . \label{eq:rp} \eeq
Here $\Rpc$ denotes the complement of set $\Rp$.

\begin{lemma}\label{lem:rpc}
Under (A2),(A3),(A4),  the following hold for $\Rp$ and $\yp$ defined in (\ref{eq:rp}):
\\
(i) $\Rpc = \{y: y \geq \yp\}$.
(ii) $\pi \gr \bp \implies \Rp  \subset \Rbp$.
(iii) $\pi \gr \bp \implies \yp  < \ybp$.
\end{lemma}
\begin{proof}
The first assertion says that the set of observations for continuing is the set $\{y: y \geq \yp\}$.  By (A4), $C_{\bu} - C_0$ has decreasing elements.
Since $T(\pi,y,u;\model)$ is MLR increasing in $y$, clearly  $(C_{\bu} - C_0)^\p T(\pi,y,u;\model)$ is decreasing in $y$.
Therefore, there exists a $\yp$ such that $y \geq \yp$ implies $T(\pi,y,u;\model) \in \Rpc$. This proves the first statement.
 By (A4), $C_{\bu} - C_0$ has decreasing elements. By (A2), (A3), $T(\pi,y,u;\model)$ is MLR increasing in $\pi$.
Therefore  $(C_{\bu} - C_0)^\p T(\pi,y,u;\model) \geq  (C_{\bu} - C_0)^\p T(e_X,y,u;\model)$ 
which implies $\Rp  \subset \Rbp$.
Statement (i) says that  $T(\pi,y,u;\model)$ is MLR increasing in $y$;
statement (ii) says that  $\Rp  \subset \Rbp$. Combining these yields  $\yp  \leq \ybp$.
\end{proof}

For notational convenience denote the optimal strategy $\mu^*(\theta)$  as $\mu$.
From (\ref{eq:cost}), the total cost incurred by applying strategy $\mu(\pi)$ to model $\model$ satisfies at time $n$
\begin{align*}
\tcost{n}(\pi;\model) &= C_{\mu(\pi)}^\p \pi + \sum_{y\in \Y} \tcost{n-1}(T(\pi,y,\mu(\pi);\model) \sigma(\pi,y,\mu(\pi);\model)\\
& = C_{\mu(\pi)}^\p \pi + \sum_{y\in \Rpc} \tcost{n-1}(T(\pi,y,\mu(\pi);\model) \sigma(\pi,y,\mu(\pi);\model) \end{align*}
since for $y \in \Rp$, $T(\pi,y,\mu(\pi);\model) \in \Stopb$ and so $V(T(\pi,y,\mu(\pi);\model)) = 0$.

Therefore, the absolute difference in total costs for models $\model, \bmodel$ satisfies
\begin{align}
|\tcost{n}(\pi;\model) - \tcost{n}(\pi;\bmodel)|  \leq &  \sum_{y \in \Rpc \cup \Rpbc} |
\tcost{n-1}(T(\pi,y,\mu(\pi);\model)  - \tcost{n-1}(T(\pi,y,\mu(\pi);\bmodel)| \,\sigma(\pi,y,\mu(\pi);\model) \nonumber\\
& \quad + \sum_{y \in \Rpc \cup \Rpbc} \tcost{n-1}(T(\pi,y,\mu(\pi);\bmodel)\, \bigl| \sigma(\pi,y,\mu(\pi);\model) - \sigma(\pi,y,\mu(\pi);\bmodel) \bigr| \nonumber \\
 \leq & 
\sup_{\pi \in \I} |\tcost{n-1}(\pi;\model)  - \tcost{n-1}(\pi;\bmodel) | \,
\sum_{y \in \Rpc \cup \Rpbc} \sigma(\pi,y,\mu(\pi);\model) \nonumber \\
& + \sup_{\pi \in \I} \tcost{n-1}(\pi;\bmodel) \sum_{y \in \Y} \bigl| \sigma(\pi,y,\mu(\pi);\model) - \sigma(\pi,y,\mu(\pi);\bmodel) \bigr|
\label{eq:ineqbd}
\end{align}
We will upper bound the various terms on the RHS of (\ref{eq:ineqbd}).
Statement (i) of Lemma \ref{lem:rpc} yields $ \Rpc \cup \Rpbc = \{y \geq \yps\}$ where $\yps = \min(\yp,\ypm)$.
Next Statement (iii) of Lemma \ref{lem:rpc} yields $\ypsx \leq \yps$.
Therefore, 
\begin{align*}
 \sup_{\pi \in \I}  |\tcost{n-1}(\pi;\model)  - \tcost{n-1}(\pi;\bmodel) |& \,
\sum_{y \in \Rpc \cup \Rpbc} \sigma(\pi,y,\mu(\pi);\model) \\
\leq &  \sup_{\pi \in \I} |\tcost{n-1}(\pi;\model)  - \tcost{n-1}(\pi;\bmodel) | \,
\max_u \sum_{y \geq \ypsx} \sigma(\pi,y,u;\model)  \\
\leq & \sup_{\pi \in \I} |\tcost{n-1}(\pi;\model)  - \tcost{n-1}(\pi;\bmodel) | \, \max_u \sum_{y \geq \ypsx} \sigma(e_X,y,u;\model) 
\end{align*}
where the last line follows since $e_X \gs \pi$, and so Statement 2 of Theorem \ref{thm:filter} implies $\sigma(\pi,\cdot,u;\model) \ls \sigma(e_X,\cdot,u;\model)$.
Also evaluating $ \sigma(\pi,y,\mu(\pi);\model) = \mathbf{1}_X^\p B_y  ({\tp^\p})^{\mu(\pi)} \pi$ defined in (\ref{eq:hmm}) yields
\begin{multline}
 \sum_{y \in \Y} \bigl| \sigma(\pi,y,\mu(\pi);\model) - \sigma(\pi,y,\mu(\pi);\bmodel) \bigr| 
\leq \max_u \sum_y \sum_i \sum_j |B_{jy} \tp^u|_{ij} - \bB_{jy} \btp^u |_{ij}| \pi(i) \\
\leq   \max_u \max_i  \sum_y \sum_j  |B_{jy} \tp^u|_{ij} - \bB_{jy} \btp^u |_{ij}|  \label{eq:modeldiff}
\end{multline}
Finally, $\sup_{\pi \in \I} \tcost{n-1}(\pi;\bmodel) \leq \max_{i\in \X}C(e_i,0)$. Using these bound in (\ref{eq:ineqbd}) yields
\beq
\sup_{\pi\in \I} |\tcost{n}(\pi;\model) - \tcost{n}(\pi;\bmodel)| \leq 
\rho\,  \sup_{\pi \in \I} |\tcost{n-1}(\pi;\model)  - \tcost{n-1}(\pi;\bmodel) | + \max_{i\in \X}C(e_i,0) \|\model - \bmodel\|
\label{eq:unravel}
 \eeq
 where $\rmm  = \max_u \sum_{y \geq \ypsx} \sigma(e_X,y,u;\model) $ and $\|\model - \bmodel\|$ is given by (\ref{eq:modeldiff}).
Since $\max_u \sum_{y \in \Y}  \sigma(e_X,y,u;\model)  = 1$, then (A7) implies    $\rmm  = \max_u \sum_{y \geq \ypsx} \sigma(e_X,y,u;\model)  < 1$.
Then starting with  $\tcost{0}(\pi;\model) = \tcost{0}(\pi;\bmodel)=0$, unravelling (\ref{eq:unravel}) yields
(\ref{eq:sens}).

{\bf Proof of Corollary \ref{cor:kl}}:  When $\model$ and $\bmodel$ have  identical transition matrices, then
(\ref{eq:modeldiff}) becomes
$$   \max_u \max_i \sum_j   \tp^u|_{ij}  \sum_y   |B_{jy}  - \bB_{jy} |   $$
From Pinsker's inequality \cite{CT06}, the total variation norm is bounded by Kullback-Leibler distance $D$ defined in (\ref{eq:kl}) as
 $$  \sum_y   |B_{jy}  - \bB_{jy} |    \leq  \sqrt{2 \, D(B_j \| \bB_j )}$$

\subsection{Proof of Theorem \ref{thm:filter}} \label{sec:tfilter}

We quote the following result from \cite{HS84}, which adapted to our notation reads
\begin{theorem}[{\cite[Lemma 8.2, pp.382]{HS84}}]  (i) Suppose $p(\cdot)$ and $q(\cdot)$ are integrable functions on $\Y$
and $f(\cdot)$ is increasing and non-negative. Then $\sum_Y f(y) p(y) \leq \sum_Y f(y) q(y)$ iff
$\sum_{y\geq \bar{y}} p(y) \geq \sum_{y\geq \bar{y}} q(y) $.
\\
(ii) Suppose $f_i$ is increasing for $i\in X$ and non-negative. Then for arbitrary vectors $p,q \in \reals^X$,
$ f^\p p \geq f^\p q  \text{ iff } \sum_{j\geq \bar{j}} p_j \geq \sum_{j \geq \bar{j}} q_j \text{ for all } \bar{j} \in \X $
\label{thm:hs} \qed
\end{theorem}
The above theorem is similar to Statement (ii) of Theorem \ref{res1} with some important difference.
Unlike Theorem \ref{res1},  $p$ and $q$ need not be probability measures.
On the other hand, Theorem \ref{res1} does not require $f$ to be non-negative.

{\bf Proof of Theorem \ref{thm:filter}}:
Statements 1, 2 and 4 of the theorem are proved in \cite{Lov87}.

Statement 3: 
Suppose  $\pi \gr \bp$. Then clearly  (A5)-(i) implies that 
$$ \sum_{j \geq q} \sum_i \left( \tp^{D_{u+1}}\vert_{ij} - \tp^{D_u}\vert_{ij}\right) \pi(i) \leq 
 \sum_{j \geq q} \sum_i  \left( \tp^{D_{u+1}}\vert_{ij} - \tp^{D_u}\vert_{ij}\right ) \bp(i). $$
Also (A3) implies that  
 $\sum_{y\geq q} B_{jy} $ is increasing in $j$. 
 Then applying Theorem \ref{thm:hs}(i)  yields
$$\sum_j \sum_{y\geq q} B_{jy} \sum_i \left( \tp^{D_{u+1}}\vert_{ij} - \tp^{D_u}\vert_{ij}\right) \pi(i) \leq
 \sum_j \sum_{y\geq q} B_{jy}\sum_i  \left( \tp^{D_{u+1}}\vert_{ij} - \tp^{D_u}\vert_{ij}\right ) \bp(i) .$$

Statement 5(a): The proof is as follows: By definition
${\tp}^\p \pi \gr {\btp}^\p \pi$ is equivalent to
$$\sum_{i\in \X} \sum_{m\in\X} \left(\tp_{ij} \btp_{m,j+1} - \btp_{ij}\tp_{m,j+1}\right) \pi_i \pi_m \leq 0. $$
Thus clearly (\ref{eq:mor}) is a sufficient condition for ${\tp}^\p \pi \gr {\btp}^\p \pi$.

Statement 5(b): Since $\tp \succeq \btp$  implies $\tp^\p  \pi \gr \btp^\p  \pi$ it follows from (A2) that $\tp^\p  \tp^\p  \pi \gr \tp^\p  \btp^\p \pi$.
Also Statement 4(a) implies  $\tp^\p \btp^\p \pi  \gr  \btp^\p \btp^\p \pi$.
Since the MLR order is transitive, these inequalities imply $\tp^\p  \tp^\p  \pi \gr  \btp^\p \btp^\p \pi$.
Continuing similarly, it follows that for any positive integer $l$, $({\tp^l})^\p \pi \gr ({\btp}^l)^\p \pi $ .

Statement 6(a):  This follows trivially since Bayes' rule preserves MLR dominance.
That is $\pi \gr \bp$ implies $\frac{B_y \pi}{\ones^\p B_y \pi} \gr   \frac{B_y \bp}{\ones^\p B_y \bp}$.
Since by Statement 4(a), $\tp \succeq \btp$  implies $\tp^\p  \pi \gr \btp^\p  \pi$, applying the Bayes rule preservation of MLR dominance proves the result.

Statement 6(b): 
 (iii) Since $\tp \succeq \btp$  implies $\tp^\p \pi \gr \btp^\p \pi$, it follows that $\tp^\p \pi \gs \btp^\p \pi$. Next  (A3) implies that
 $\sum_{y\geq q} B_{iy} $ is increasing in $i$.  Therefore
 $\sum_{i\in\X} \sum_{y\geq q} B_{iy} [\tp^\p \pi](i) \geq \sum_{i\in \X} \sum_{y\geq q}B_{iy} [\btp^\p \pi](i)$.

\bibliographystyle{plain}

\bibliography{$HOME/styles/bib/vkm}
\end{document}